\documentclass[UTF8]{amsart}
\usepackage{graphicx}
\usepackage[rotateright]{rotating}
\usepackage{subfigure,fancybox}
\usepackage{amscd,amstext}
\usepackage{amsmath}
\usepackage{amssymb}
\usepackage{mathrsfs}
\usepackage{color}
\usepackage{multirow}
\usepackage{bm}
\usepackage{algorithm}
\usepackage{multirow}
\usepackage{algorithmicx}
\usepackage{algpseudocode}
\usepackage{float}

\usepackage{amsthm}

\newcommand{\Rmnum}[1]{\expandafter\@slowromancap\romannumeral #1@}
\makeatother
\newtheorem{tpcl}{Tpcl}[section]
\newtheorem{theorem}[tpcl]{Theorem}
\newtheorem{lemma}[tpcl]{Lemma}

\newtheorem{example}[tpcl]{Example}
\newtheorem{assumption}[tpcl]{Assumption}
\newtheorem{remark}[tpcl]{Remark}

\renewcommand{\algorithmicrequire}{\textbf{Input:}}

\begin{document}

\title{Shape optimization problems with random coefficients via the penalty method}

\author{Xiaowei Pang}

\address{ Department of Mathematics, Hebei Normal University,
Shijiazhuang, Hebei, China. }
\address { Hebei Key Laboratory of Computational Mathematics and Applications, Hebei.
 }

\thanks{Corresponding Author: Xiaowei Pang (pangxw21@hebtu.edu.cn)}

\maketitle


\begin{abstract}
For  shape  optimization problems, governed by elliptic equations with Dirichlet boundary condition and random coefficients, we utilize a penalization technique to get the approximate  problem.  We consider that uncertainties exists in the diffusion coefficients and minimize objective functions in mean value form. Finite element method, Monte Carlo method and accelerated version of the gradient descent method are applied to solve the corresponding  discretized problem. The convergence analysis and numerical results are included.

\end{abstract}





\section{Introduction}
Shape and topology optimization have emerged as pivotal tools in engineering design \cite{P84,SISG13,UB08}. They systematically  improve  structural performance under given constraints, such as minimizing compliance, maximizing stiffness, or enhancing reliability. Modeling  physical phenomena via partial differential equations (PDEs) is central to this process. Elliptic PDEs are key tool for describing equilibrium states in solid mechanics, heat transfer, and fluid flow \cite{AD15,CCL10,NST06}. However, real-world systems exhibit  inherent  uncertainties. These arise from material properties, applied loads, or geometric variations. Such uncertainties  render deterministic PDE models insufficient. Consequently,  shape optimization problems constrained by stochastic elliptic PDEs, where randomness is incorporated into coefficients, source terms, or boundary data, attracts significant interest \cite{GGZ23,LWWW08,MKP15}.

In this work, we concentrate on  solving stochastic shape optimization problem constrained by random elliptic PDEs,
in which the optimization variable is deterministic, and state variable is uncertain, by minimizing the expected value of the cost functional over all random parameters. We utilize the discretize-then-optimize framework to achieve our aim \cite{HPUU10}. 

A key challenge in shape optimization lies in handling the dynamic nature of the domain. As the shape or topology evolves, the computational mesh must adapt. This adaption causes remeshing difficulties, numerical instability, and increased computational cost. To address this,  penalty methods for the constraint equation and cost functional have gained prominence \cite{KKM16,MT19,NT08,PT13,PPT13}. These methods map the evolving physical domain onto a fixed reference domain. They  introduce a level set function to represent shape or topology changes within a static computational framework. By avoiding mesh deformation or regeneration, it simplifies the implementation of numerical schemes. This simplification is especially useful when coupled with finite element discretizations of PDEs. The approach is particularly advantageous for problems with homogeneous boundary conditions—such as Dirichlet or Neumann conditions with zero values on fixed boundaries. Here, the fixed domain can be preprocessed to enforce boundary constraints consistently across all design iterations \cite{NT08,NPT09}. In \cite{B23}, Bogosel  generated the discrete convex shape with the aid of  support function or the gauge function. Chakib etal. solved the shape optimization problem governed by Laplace or Stokes operator using the
Minkowski deformation of convex domain \cite{CKS24}. Other numerical methods  include  multigrid method, continuity preserving discrete shape gradient method, level set method, and neural work method, the reader can  refer to \cite{ABV13,GLZ22,MHKP16,WYZY24}.

Coupling randomness with shape optimization in stochastic elliptic PDEs adds significant complexity. Uncertainties in, for instance, the diffusion coefficient or the source term of the elliptic equation propagate the solution. This propagation affects the objective function, like expected compliance or failure probability.   Approaches to handle such uncertainties,  include using deterministic approximations. Such as  \cite{AD15} applied Taylor expansions to the cost function with respect to random variables. These approximations  reduce the computational burden and make stochastic shape optimization feasible. Conti et al. studied  risk averse elasticity optimization  problem based on a regularized gradient flow in \cite{CHPRS11}, they minimized a nonlinear monotone increasing function of the objective function  to replace 
the expectation of the objective function.  More literature on random shape optimization problems can be found in \cite{DDH15,GLW23,GGZ23,MKP15,MHKP18}.   Beyond that, Monte Carlo simulations,  although often suffer from high computational cost, but can be applied to almost any complex problem, especially excelling in handling high-dimensional integrals, complex stochastic systems, and probability calculations. Besides, the computations of different samples are completely independent, which can perfectly utilize parallel computing resources and achieve efficient acceleration. 

 Ultimately, in shape optimization problems, most studies employ gradient descent or quasi-Newton algorithms. These algorithms leverage shape derivatives to guide design updates. For example,  Martínez-Frutos et al. computed sensitivities of the objective function with respect to shape perturbations to drive robust optimal design, \cite{MKP15}. In \cite{B23}, Bogosel used gradient-based approaches for numerical shape optimization among convex sets, ensuring convergence within convex design spaces.  Chakib et al.  designed numerical approaches---augmented Lagrangian and adaptive schemes  to handle constraints in convex domain shape optimization, improving efficiency and convergence \cite{CKS24}.


 This paper is organized as follows.  In Section \ref{sec2}, we recall some basic definitions, preliminaries about random shape optimization problem, and present the convergence result under some conditions. Moving on to Section \ref{sec3}, we propose the numerical algorithm. Some numerical examples are reported in Section \ref{sec4}. Finally, Section \ref{sec5} gives the conclusion.

\section{RPDEs-constrained optimal shape problem}\label{sec2}
\subsection{Preliminaries}\label{ssec:pre}
 For the convenience of subsequent descriptions, we present
 some notions and definitions firstly. Let $D$ is a fixed convex bounded connected
 polygonal domain in $\mathbb{R}^{2}$, $v$ is a measurable function, following the symbols used in \cite{HPUU10},
 the Lebesgue integration function space on $D$ can be defined by
 \begin{equation*}
 L^{p}(D):=\left\{v~\big|~ \|v\|_{L^{p}(D)}<\infty,~1\leq p\leq \infty\right\}
 \end{equation*}
 with
 \begin{eqnarray*}
 \begin{aligned}
 &\|v\|_{L^{p}(D)}=\left(\int_{D}|v(x)|^{p}dx\right)^{\frac{1}{p}},~~p\in[1,\infty),\\
 &\|v\|_{L^{\infty}(D)}=\mbox{ess}\sup_{x\in D}|v(x)|=\inf\{\alpha\geq
 0,~\mu(\{|v(x)|>\alpha\})=0\},
 \end{aligned}
 \end{eqnarray*}
 where $\mu(\cdot)$ stands for the measure. Therefore,
 the corresponding Sobolev space could be described by
 \begin{equation*}
 W^{m,p}(D)=\left\{v~\big|~D^{\alpha}(v)\in L^{p}(D),~|\alpha|\leq
 m\right\}
 \end{equation*}
 with the norm
 \begin{eqnarray*}
 \begin{aligned}
 &\|v\|_{W^{m,p}}=\left(\sum\limits_{|\alpha|\leq
 m}\|D^{\alpha}v\|^{p}_{L^{p}(D)}\right)^{1/p},~~ p\in[1,\infty),\\
 &\|v\|_{W^{m,\infty}(D)}=\sum\limits_{|\alpha|\leq
 m}\|D^{\alpha}v\|_{L^{\infty}(D)}.
 \end{aligned}
 \end{eqnarray*}
  Here, $D^{\alpha}v$ stands for the weak derivative of $v$.
  Further, it is not difficult to find that
  $H^{k}(D)=W^{k,2}(D)$ is a Hilbert space with the inner product
   \begin{equation*}
   (u,v)_{H^{m}(D)}=\sum\limits_{|\alpha|\leq
   m}\left(D^{\alpha}u,D^{\alpha}v\right)_{L^{2}(D)}.
   \end{equation*}
  
   Next, let $H_{0}^{m}(D)=W_{0}^{m,2}(D)$ denotes the closure of
   $C_{0}^{\infty}(D)$ in $W^{m,2}(D)$, and $H^{-1}(D)=(H_{0}^{1}(D))^{*}$
   be the dual space of $H_{0}^{1}(D)$. Applying Poincar$\acute{e}$'s
   inequality \cite{T10}, there must exist a Poincar$\acute{e}$ constant $C_{p}$ such that
   \begin{equation}\label{eq:poincare}
   |v|_{H^{1}(D)}\leq\|v\|_{H^{1}(D)}\leq C_{p}|v|_{H^{1}(D)},~~\forall~v\in H^{1}_{0}(D),
   \end{equation}
 where $|v|_{H^{1}(D)}=\int_{D}|\nabla v|^{2}dx$.
  
The Bochner function space should be introduced as follows:
   \begin{equation*}
   L^{p}(\Gamma;H)=\left\{v~\Big|~\int_{\Gamma}\|v(\cdot,\omega)\|^{p}_{H}dP<+\infty\right\},
   \end{equation*}
   where $v(\cdot,\omega)$ is measurable function from $\Gamma$ to $H$ with the  measure space $(\Gamma,\mathcal{F},P)$,
   and $H$ is a Hilbert function space, the norm of Bochner spaces is denoted by
   \begin{equation*}
   \|v\|_{L^{p}(\Gamma;H)}=\left(\int_{\Gamma}\|v(\cdot,\omega)\|^{p}_{H}dP\right)^{\frac{1}{p}},\quad 1\leq p<\infty.
   \end{equation*}
     At the end of this subsection, we introduce Hedberg–Keldys stability property for domains of class $C$,   \cite[see Chapter 2 ]{NST06}.
     \begin{lemma}\label{lem:HK}
    If  $\Omega$ is an open bounded set of class $C$,  $z\in H^1(\mathbb{R}^d)$ and $z=0$ almost everywhere in $\mathbb{R}^d\backslash\Omega$, then $z\in H_{0}^{1}(\Omega)$.
     \end{lemma}
 \subsection{Statement of shape optimization problem}
Based on the notations defined in Subsection \ref{ssec:pre}, the problems of interest are  shape optimization problem  constrained by random  elliptic PDEs (RPDEs):
\begin{equation}\label{eq:ex obj}
\min_{K}\mathbb{E}\left[\int_{\Omega}J(x,\omega,u(x,\omega))\right]
\end{equation}
\begin{equation}
-\nabla\cdot\left(\alpha(x,\omega)\nabla u(x,\omega)\right)=f(x),\quad in\quad K,
\end{equation}
\begin{equation}
u(x,\omega)=0,\quad on \quad \partial K.
\end{equation}
Here, $K$ is unknown  domain with  $K\subset D\subset\mathbb{R}^d$ or there exists a given open bounded set $O$ such that $O\subset K\subset D\subset\mathbb{R}^d$,  and $\Omega$ is a integration domain, it could be the domain $K$, $O$. Besides, state variable  $u\in U=L^{2}(\Gamma;H_{0}^{1}(D))$, and right-hand side term  $f\in F=L^{2}(D)$ is a deterministic function.  The notation $\mathbb{E}$ denotes the expectation operator defined by $\mathbb{E}[z(\omega)]=\int_{\Gamma}z(\omega)dP$, where $\omega$ is a random
variable of probability space $(\Gamma,\mathcal{F},P)$. The function $J(\cdot,\cdot,\cdot)$ is some measurable and continuous Carath\'{e}odory mapping. For example, in  `optimal layout' problem, $J=\|y-y_d\|^2$, and $y_d\in L^2(D)$ is a given function  \cite[Chapter 5]{NST06}, while $J=a(x)\|y-y_d\|^2+b(x)\|\nabla y-z_d\|^2$ in tomography problem, where $a(x),b(x)\in L_{+}^\infty(D)$ and $z_d\in L^{2}(D)^d$ are given functions \cite{NPT09}.

The main difference between the
studied model here and the traditional shape optimization problem is the constraint PDE, which has
an uncertainty coefficient $\alpha(x,\omega)$. 

Under the following assumption
on $\alpha$, applying Lax-Milgram theorem and Poincar$\acute{e}$ inequality, it leads to the  conclusion that the constraint stochastic
PDE is well-posed for almost all $\omega\in \Gamma$.
 \begin{assumption}\label{assum:assum1}(cf.~\cite{GP19})
  For all $(x,\omega)\in D\times \Gamma$, there exists
 $0<\alpha_{min}<\alpha_{max}<\infty$, such that
 \begin{equation}
 0<\alpha_{min}\leq \alpha(x,\omega)\leq \alpha_{max}<\infty.
 \end{equation}
 \end{assumption}
The variational form of the constraint, the stochastic
PDE is: find $u\in U$ satisfies
\begin{equation}\label{eq:weak form}
 \mathbb E\left[\bar{a}(u,v;\omega)\right]=\mathbb E\left[l(f,v)\right],~~\forall v\in U,
\end{equation}
where $\bar{a}(u,v;\omega)=\left(\alpha(x,\omega)\nabla u,\nabla v\right)_{L^{2}(D)}$
and $l(f,v)=\left(f,v\right)_{L^{2}(D)}$. For any $\omega \in \Gamma$, it is easy to see
that $\bar{a}: U\times U\rightarrow \mathbb{R}$ and $l: F\times U\rightarrow \mathbb{R}$
are all continuous bilinear functionals.  
\begin{lemma}\label{lem:well-posed of primal}(cf. \cite{Rey15})
Under the assumption \ref{assum:assum1}, the stochastic elliptic equations admit
a unique solution $u\in L^{2}(\Gamma;H_{0}^{1}(D))$ which has the
following estimate
\begin{equation}
\|u(\cdot,\omega)\|_{L^{2}(D)}\leq C_{1}\|f\|_{L^{2}(D)}, ~\text{for
a.e.}~\omega \in \Gamma,
\end{equation}
where $C_{1}=\frac{C_{p}}{\alpha_{min}}$ is a constant.
 \end{lemma}

\subsection{Approximation problem}
 We first parametrize the set $K$ by a continuous function $g$, so that $K$ can be denoted by
 \begin{equation}\label{eq:Kg}
 K=K_{g}=int\{~x\in D|~g(x)\geq 0\},
 \end{equation} 
it is not necessarily connected,  $g\in C(\bar{D})$ is the corresponding unknown parameterization, not uniquely determined by domain $K$ . 

Next,  we add a penalty term on the constraint equation and  define the approximation of state equation 
\begin{equation}\label{eq:approx state}
-\nabla\cdot\left(\alpha(x,\omega)\nabla u_{\epsilon}(x,\omega)\right)+\frac{1}{\epsilon}\left(1-H_{\epsilon}(g)\right)u_{\epsilon}(x,\omega)=f(x),\quad in\quad D,
\end{equation}
\begin{equation}\label{eq:approx boundry}
u_{\epsilon}(x,\omega)=0,\quad on \quad \partial D.
\end{equation}
 Here, $\epsilon>0$ is a penalty parameter and
 $$ H_{\epsilon}(g)=\left\{
 \begin{array}{rcl}
 &1,& {g\geq0},\\
 &\frac{(\epsilon-2g)(g+\epsilon)^2}{\epsilon^3}, &{-\epsilon<g<0},\\
 &0, & {g\leq -\epsilon},
 \end{array} \right. $$
 is a Lipschitz function.
  
   Naturally, we can get the variational  form of the approximate constraint equation: find $u_{\epsilon}\in U$ such that 
\begin{equation}\label{eq:approx weak form}
 \mathbb E\left[a(u_{\epsilon},v;\omega)\right]=\mathbb E\left[l(f,v)\right],~~\forall~v\in U,
\end{equation}
where $a(u_{\epsilon},v;\omega)=\left(\alpha(x,\omega)\nabla u_{\epsilon},\nabla v\right)_{L^{2}(D)}+\frac{1}{\epsilon}\left((1-H_{\epsilon}(g))u_{\epsilon},v\right)_{L^2(D)}$.

Similar to Lemma~\ref{lem:well-posed of primal},  for any $\omega \in \Gamma$, the boundedness  of $H_{\epsilon}(g)$ can guarantee that the weak form is well-posed. Under appropriate assumptions on the domain $K$, we extend the result of the deterministic  situation into a  convergence result in the sense of expectation \cite{NPT09}.
 
  \begin{theorem}
   If $K=K_{g}$ is Lipschitz domain, then there exists subsequences of $\mathbb{E}[u_{\epsilon}]$ that weakly  converge to $\mathbb{E}[\bar{u}]$ in $H^1(K_{g})$ and strongly in $L^2(K_{g})$.
   \end{theorem}
   \begin{proof}
    For any test function $v=u_{\epsilon}(\cdot,\omega)$, based on Poincar$\acute{e}$ inequality (\ref{eq:poincare}) and Assumption~\ref{assum:assum1}, we obtain
    \begin{equation}
    \alpha_{min}\|u_{\epsilon}(\cdot,\omega)\|_{H^{1}_{0}(D)}^{2}+\frac{1}{\epsilon}\int_{D}\left(1-H_{\epsilon}(g)\right)u_{\epsilon}^2(\cdot,\omega)dx\leq\int_{D}f(x)u_{\epsilon}(\cdot,\omega)dx,
    \end{equation} 
 For all of the random variable, it yields
    \begin{eqnarray}
    \alpha_{min}\mathbb{E}\left[\|u_{\epsilon}\|^{2}\right]+\frac{1}{\epsilon}\int_{\Gamma}\int_{D}\left(1-H_{\epsilon}(g)\right)u_{\epsilon}^2dxdP\leq\int_{\Gamma}\int_{D}f(x)u_{\epsilon}dxdP.
    \end{eqnarray}
    It yields that $\{u_{\epsilon}\}$ is bounded in $U$, due to the positivity of the penalization term. On one hand, we claim that $\{\mathbb{E}[u_{\epsilon}]\}$ is bounded in $H^1(K_g)$, this is owing to $K_g \subset D$  and
        \begin{equation*}
        \|\mathbb{E}[u_{\epsilon}]\|_{H^{1}(K_{g})}\leq\mathbb{E}\left[\|u_{\epsilon}\|_{H^{1}(D)}\right]\leq \sup_{\epsilon}\|u_{\epsilon}\|_{U}< \infty.
        \end{equation*}
        There  exists a subsequence  also denoted by $\{\mathbb{E}[u_{\epsilon}]\}$, which  weakly converge to $\mathbb{E}[\bar{u}]$  in $H^{1}(K_{g})$.  
        
        On the other hand, since $K_{g}$  is Lipschitz and it is also  class $C$. As mentioned above, $\{\mathbb{E}[u_{\epsilon}]\}$ is  uniformly bounded in $H^1(K_g)$, by Rellich-Kondrachov theorem \cite[Theorem 6.3, Remarks 6.4]{AF03}, then the embedding $H^{1}(K_{g})\to L^{2}(K_{g})$ is compact and there exists  a subsequence  also noted by $\mathbb{E}[u_{\epsilon}]$, which    strongly converges to   some $w\in L^{2}(K_{g})$. While the weakly convergence  $\mathbb{E}[u_{\epsilon}]\rightarrow\mathbb{E}[\bar{u}]$ holds in $H^{1}(K_{g})$. By the uniqueness of the limit, $w=\mathbb{E}[\bar{u}]$, namely, $\mathbb{E}[u_{\epsilon}]\rightarrow \mathbb{E}[\bar{u}]$ in $L^{2}(K_{g})$,  when $\epsilon\rightarrow 0$.
   
 Furthermore,
        \begin{equation*}
        \int_{\Gamma}\int_{D}(1-H_{\epsilon}(g))u^2_{\epsilon}dxdP\rightarrow 0, 
        \end{equation*}
as $\epsilon\rightarrow 0$. Notice that  function $H_{\epsilon}(g), g$ are continuous, $g(x)\leq 0$ holds in $D\backslash K_{g}$,   so for sufficiently small $\epsilon$, any compact set $S\subset D\backslash K_{g}$ and $x\in S$, there exists a constant $r>0$ enables $g(x)\leq -r$ and $H_{\epsilon}(g)=0$ in the domain $S$. Thus, when $\epsilon$ goes to $0$, $u_{\epsilon}$ also goes to $0$  in $L^2(\Gamma;L^2(S))$,  by the arbitrariness of $S$,  $u_{\epsilon}\rightarrow 0$ in $L^{2}(\Gamma;L^{2}(D\backslash K_{g}))$, thereby $\bar{u}(x,\omega)=0$ holds a.e. in $D\backslash K_{g}\times \Gamma$ and $\mathbb{E}[\bar{u}]=0$ a.e. in $D\backslash K_{g}$. 
   
   Moreover,   for a.e. $\omega\in \Gamma$, by zero extension to $\mathbb{R}^{d}\backslash D$,
   \begin{equation*}
 \bar{u}(x, \omega) = 
 \begin{cases} 
 \bar{u}(x, \omega) & \text{if } x \in D, \\
 0 & \text{if } x \in \mathbb{R}^d \setminus D,
 \end{cases}
   \end{equation*}
   $D\backslash K_{g}\subset\mathbb{R}^{d}\backslash K_{g}$ illustrates $\bar{u}=0$ a.e. in $\mathbb{R}^{d}\backslash K_{g}$.
   Thus, $\bar{u}(x,\omega)=0$ a.e. in $(\mathbb{R}^{d}\backslash K_{g})\times \Gamma$, combining Lemma~\ref{lem:HK} (with $K_{g}$ of class C) imply $\bar{u}(\cdot,\omega)\in H_{0}^{1}(K_{g})$, hence $\mathbb{E}[\bar{u}]\big|_{K_{g}}\in H_{0}^{1}(K_{g})$.
    
    For any test function $v\in L^{2}(\Gamma;C_{0}^{\infty}(K_{g}))$, $H_{\epsilon}(g)=1$, and the approximate weak form is 
    \begin{equation}
    \int_{\Gamma}\int_{K_{g}}\alpha(x,\omega)\nabla u_{\epsilon}\nabla vdxdP=\int_{\Gamma}\int_{K_{g}}fvdxdP,
    \end{equation}
    when $\epsilon$ tends to $0$, the above equality indicates $\mathbb{E}[\bar{u}]\big|_{K_{g}}\in H_{0}^{1}(K_{g})$ is the weak solution of the original weak form (\ref{eq:weak form}).
   \end{proof}
   In (\ref{eq:ex obj}), if  
   \begin{equation}\label{eq:E obj} \int_{\Omega}J(x,\omega,u(x,\omega))=\frac{1}{2}\int_{K}(u-u_{d})^{2}dx,
   \end{equation}
    through similar discussion with \cite{HPUU10},  it  is easy to get the weak form of dual or adjoint equation of (\ref{eq:approx weak form}): finding $z\in U$, such that
   \begin{equation}\label{eq:dual weak form}
   \int_{\Gamma}\int_{D}\left(\alpha(x,\omega)\nabla z\nabla v+\frac{1}{\epsilon}(1-H_{\epsilon}(g))zv\right)dxdP=\int_{\Gamma}\int_{K}(u-u_{d})vdxdP,\quad \forall~ v\in U.
   \end{equation}
   Under the consideration of the parametrization process, we can notice that (\ref{eq:E obj}) is equivalent to \begin{equation}\label{eq:D obj}
   \frac{1}{2}\int_{D}H_{\epsilon}(g)\left(u-u_{d}\right)^{2}dx,
   \end{equation}
    and  the right-hand side of (\ref{eq:dual weak form}) can be replaced by $\int_{\Gamma}\int_{D}H_{\epsilon}(g)(u-u_{d})vdxdP$.
    \begin{remark}
    If  
       \begin{equation}\label{eq:O obj} \int_{\Omega}J(x,\omega,u(x,\omega))=\frac{1}{2}\int_{O}(u-u_{d})^{2}dx,
       \end{equation}
       then the right-hand side of the dual weak form  still is $\int_{\Gamma}\int_{O}H_{\epsilon}(g)(u-u_{d})vdxdP$ for the  $O$ is a given bounded region, it will not appear $H_{\epsilon}(g)$,  and objective function is the same. 
    \end{remark}
   Besides, the weak form (\ref{eq:dual weak form}) also has well-posedness.  On account of the uniqueness of the weak solution, we can compute the gradient of the objective functional.
   \begin{lemma}\label{lem:gradientsimdual}
   The stochastic approximate weak solution $u_{\epsilon}(g,\omega)$ of (\ref{eq:approx weak form}) is G$\hat{a}$teaux differentiable about $g$ from $X(D)$ to $L^{2}(\Gamma;H_{0}^{1}(D))$ and for any $q\in X(D)$,  $y=\nabla u_{\epsilon}(g,\omega)q$ satisfies the following dual weak form
   \begin{equation}\label{eq:simi dual}
   \int_{\Gamma}\int_{D}\left(\alpha(x,\omega)\nabla y\nabla v+\frac{1}{\epsilon}(1-H_{\epsilon}(g))yv\right)dxdP=\frac{1}{\epsilon}\int_{\Gamma}\int_{D}H_{\epsilon}'(g)qu_{\epsilon}vdxdP,
   \end{equation}
    where   $X(D)$ is a functional space  in $D$.
   \end{lemma}
   \begin{proof}
   For fixed $g, q \in X(D)$ and $\delta \neq 0$, let $u_{\epsilon}^{\delta}(\omega) = u_{\epsilon}(g + \delta q,\omega)$, $u_{\epsilon}(\omega) = u_{\epsilon}(g,\omega)$ are  solutions of  (\ref{eq:approx weak form}) corresponding to the parameter $g+\delta$ and $g$, separately. For any $v\in L^{2}(\Gamma;H_{0}^{1}(D))$, subtracting the state equations under the parameters  $g + \delta q$ and $g$, we get
   \begin{small}
   \begin{equation*}
    \int_{\Gamma}\int_{D}\alpha(x,\omega)\nabla (u_{\epsilon}^{\delta}-u_{\epsilon})\nabla vdxdP+\frac{1}{\epsilon}\int_{\Gamma}\int_{D}\left[(1-H_{\epsilon}(g+\delta q))u_{\epsilon}^{\delta}-(1-H_{\epsilon}(g))u_{\epsilon}\right]vdxdP=0.
     \end{equation*}
     \end{small}
   Define $z^{\delta}:=u_{\epsilon}^{\delta}(\omega)-u_{\epsilon}(\omega)\in L^{2}(\Gamma;H_{0}^{1}(D))$,  the above equality can be reformulated by
   \begin{eqnarray}\label{eq:substract}
   \int_{\Gamma}\int_{D}\alpha(x,\omega)\nabla z^{\delta}\nabla vdxdP+\frac{1}{\epsilon}\int_{\Gamma}\int_{D}(1-H_{\epsilon}(g+\delta q))z^{\delta}vdxdP\\
   =\frac{1}{\epsilon}\int_{\Gamma}\int_{D}\frac{H_{\epsilon}(g+\delta q)-H_{\epsilon}(g)}{\delta}u_{\epsilon}vdxdP\nonumber.
   \end{eqnarray}
   Choosing $v=z^{\delta}$, it obtains 
   \begin{eqnarray*}
   \int_{\Gamma}\int_{D}\alpha(x,\omega)|\nabla z^{\delta}|^{2}dxdP+\frac{1}{\epsilon}\int_{\Gamma}\int_{D}(1-H_{\epsilon}(g+\delta q))|z^{\delta}|^{2}dxdP\\
   =\frac{1}{\epsilon}\int_{\Gamma}\int_{D}\frac{H_{\epsilon}(g+\delta q)-H_{\epsilon}(g)}{\delta}u_{\epsilon}z^{\delta}dxdP.
   \end{eqnarray*} 
By using the Assumption~\ref{assum:assum1}, Poincar$\acute{e}$  inequality, Cauchy-Schwarz inequality, and notice that $H_{\epsilon}(g)$ is Lipschitz continuous, we conclude that the sequence $\{z^{\delta}\}$ is bounded in $L^{2}(\Gamma;H_{0}^{1}(D))$. Therefore, there exists  subsequences $\delta_{k}\rightarrow 0$  and $y \in L^{2}(\Gamma;H_{0}^{1}(D))$ such  $z^{\delta_k} \rightarrow y$  is weakly convergent in $L^{2}(\Gamma;H_{0}^{1}(D))$, i.e., for any $\phi \in L^{2}(\Gamma;H^{-1}(D))$,
\begin{equation*}
\int_{\Gamma}\langle\phi(\omega),z^{\delta_{k}}(\omega)\rangle dP\rightarrow \int_{\Gamma}\langle\phi(\omega),y(\omega)\rangle dP,
\end{equation*}
and $y$ is independent of the subsequence $z^{\delta_{k}}$.  Hence, $z^{\delta} \rightarrow y$ is weakly convergent. This illustrates the first conclusion.

 Let $\delta\rightarrow 0$ in  (\ref{eq:substract}), since $z^{\delta} \rightarrow y$ is weakly convergent, then 
\begin{eqnarray*}
  \int_{\Gamma}\int_{D}\alpha(x,\omega)\nabla z^{\delta}\nabla vdxdP+\frac{1}{\epsilon}\int_{\Gamma}\int_{D}(1-H_{\epsilon}(g+\delta q))z^{\delta}vdxdP\\
 \rightarrow\int_{\Gamma}\int_{D}\alpha(x,\omega)\nabla y\nabla v+\frac{1}{\epsilon}(1-H_{\epsilon}(g))yvdxdP.
 \end{eqnarray*}
   It is easy to verify that $\frac{H_{\epsilon}(g+\delta q)-H_{\epsilon}(g)}{\delta}u_{\epsilon}v$ is convergent a.e. in $\Gamma\times D$ as $\delta\rightarrow 0$. Besides, there exists a function $G(x,\omega)$ and Lipschitz constant $L$, such that
   \begin{equation*}
   \left|\frac{H_{\epsilon}(g+\delta q)-H_{\epsilon}(g)}{\delta}u_{\epsilon}v\right|\leq L\|q\|_{L^{\infty}(D)}|u_{\epsilon}||v|:=G(x,\omega).
   \end{equation*}
   Due to $u_{\epsilon},v\in L^{2}(\Gamma;L^{2}(D))$, it yields the following integrability
   \begin{equation*}
   \int_{\Gamma}\int_{D}G(x,\omega)dxdP=L\|q\|_{L^{\infty}(D)}\|u_{\epsilon}\|_{L^{2}(\Gamma;L^{2}(D))}\|v\|_{L^{2}(\Gamma;L^{2}(D))}.
   \end{equation*}
    By dominated convergence theorem for the right-hand side of (\ref{eq:substract}), we have
  \begin{equation*} \frac{1}{\epsilon}\int_{\Gamma}\int_{D}\frac{H_{\epsilon}(g+\delta q)-H_{\epsilon}(g)}{\delta}u_{\epsilon}vdxdP\rightarrow\frac{1}{\epsilon}\int_{\Gamma}\int_{D}H'_{\epsilon}(g)qu_{\epsilon}vdxdP.
  \end{equation*}

To sum up, for any $v\in L^{2}(\Gamma;H_{0}^{1}(D))$, (\ref{eq:simi dual}) holds.  
   \end{proof}
   
   \begin{theorem}
   If the objective functional has the form $\mathbb{E}\left[\frac{1}{2}\int_{K}(u-u_{d})^{2}dx\right]$, and $g\geq 0$ holds in $K$, then its directional derivative equals
   \begin{equation}
   \mathbb{E}\left[\frac{1}{\epsilon}\int_{D}(H_{\epsilon}'(g))qu_{\epsilon}zdx\right],
   \end{equation}
   here, $z$ is the weak solution of the approximate adjoint equation (\ref{eq:dual weak form}).
   \end{theorem}
  \begin{proof}
Due to the uniqueness of the weak solution of the approximate state equation,  by (\ref{eq:Kg})-(\ref{eq:approx boundry}),   with the abuse of notations,
 define the objective functional
\begin{equation} j(g)=\frac{1}{2}\mathbb{E}\left[\int_{K}(u(g)-u_{d})^2dx\right],
\end{equation}
 naturally, for any $\delta>0$ and $q\in X(D)$, 
\begin{eqnarray}
\begin{aligned}
j(g+\delta q)&=\frac{1}{2}\mathbb{E}\left[\int_{K}(u(g+\delta q)-u_{d})^2dx\right]\\
&=\frac{1}{2}\mathbb{E}\left[\int_{K}\big|u(g+\delta q)-u(g)+u(g)-u_{d}\big|^{2}dx\right]\\
&=\frac{1}{2}\mathbb{E}\left[\int_{K}\big|u(g)-u_{d}\big|^{2}dx \right]+\frac{1}{2}\mathbb{E}\left[\int_{K}\big|u(g+\delta q)-u(g)\big|^2dx\right]\\
&\quad+\mathbb{E}\left[\int_{K}(u(g+\delta q)-u(g),u(g)-u_{d})dx\right],
\end{aligned}
\end{eqnarray}
 by using the adjoint equation (\ref{eq:dual weak form}) and Lemma~\ref{lem:gradientsimdual}, we have
 \begin{eqnarray*}
 \begin{aligned}
 &\partial_{g}j(g)\delta=\lim\limits_{\delta\rightarrow 0}\frac{j(g+\delta q)-j(g)}{\delta}\\
 &=\lim\limits_{\delta\rightarrow 0}\frac{\frac{1}{2}\mathbb{E}\left[\int_{K}\big|u(g+\delta q)-u(g)\big|^2dx\right]+\mathbb{E}\left[\int_{K}(u(g+\delta q)-u(g),u(g)-u_{d})dx\right]}{\delta}\\
 &=\lim\limits_{\delta\rightarrow 0}\mathbb{E}\left[\int_{K}\frac{(u(g+\delta q)-u(g))^{2}}{2\delta}dx\right]+\mathbb{E}\left[\int_{K}\lim\limits_{\delta\rightarrow 0}\frac{u(g+\delta q)-u(g)}{\delta}(u(g)-u_{d})dx\right]\\
 &:=I+II,
 \end{aligned}
 \end{eqnarray*}
 For term $I$,  Lemma~\ref{lem:gradientsimdual} tells us $\{z^{\delta}\}$ is uniformly bounded in $L^{2}(\Gamma;L^{2}(D))$, i.e., there is a constant $C$ satisfies $\sup\limits_{\delta}\mathbb{E}\left[\|z^{\delta}\|_{L^{2}(D)}\right]\leq C$, 
 this equivalents 
 \begin{equation*}
 \left|\mathbb{E}\left[\int_{K}\frac{(u(g+\delta q)-u(g))^{2}}{2\delta}dx\right]\right|\leq \mathbb{E}\left[\frac{\delta}{2}\|z^{\delta}\|_{L^{2}(D)}\right]\leq \frac{\delta}{2}C,
 \end{equation*}
 Naturally, $I=0$ and $II=\mathbb{E}\left[\int_{K}y(u(g)-u_{d})\right]$.
 In (\ref{eq:dual weak form}), let $v=y\in L^{2}(\Gamma;H_{0}^{1}(D))$, then in virtue of (\ref{eq:simi dual}) with $v=z$, we obtain
 \begin{equation*}
 \mathbb{E}\left[y(u(g)-u_{d})dx\right]=\frac{1}{\epsilon}\mathbb{E}\left[\int_{D}H'_{\epsilon}(g)qu_{\epsilon}zdx\right].
 \end{equation*}
Consequently, $\partial_{g}j(g)\delta=\frac{1}{\epsilon}\mathbb{E}\left[\int_{D}H'_{\epsilon}(g)qu_{\epsilon}zdx\right]$, the proof is finished.
  \end{proof}
  
\section{Proposed algorithm}\label{sec3}

In this section, we first mainly focus on the discretization method (FEM) on
physical space, and then use Monte Carlo approximation (MC) to deal with random variable. Lastly,  an accelerated version of gradient descent method (AGD)  is adopted to update $g$.
\subsection{FEM discretization}\label{FEM-dis}
For applying FEM, we consider a family of
triangulations $\{\mathcal{T}_{h}\}_{h>0}$ consist of triangles $T$
which satisfy $\overline{D}=\bigcup_{T\in\mathcal{T}_{h}}T$,
and suppose all the triangulations used in this paper are shape
regular. Here $h$ stands for the maximal diameter of all triangles $T$.
With the triangulation $\{\mathcal{T}_{h}\}_{h>0}$, define the finite element
spaces 
\begin{eqnarray*}
\begin{aligned}
{V}_{h}&~=\{v_{h}\in H^{1}(D)~\big|~v_{h}|_{T}\in
\mathbb{P}_1 , ~\forall~ T\in \mathcal{T}_h\},\\
{V}^{0}_{h}&~=\{v_{h}\in H_{0}^{1}(D)~\big|~v_{h}|_{T}\in
\mathbb{P}_1 ,~ \forall~ T\in \mathcal{T}_h\},
\end{aligned}
\end{eqnarray*}
 and the approximation sets of $U$ and $F$ could be defined as:
\begin{eqnarray*}\label{FEMSpace}
\begin{aligned}
U_{h}&~=L^{2}(\Gamma;V^{0}_{h}),\quad F_{h}=V_{h},
\end{aligned}
\end{eqnarray*}
where $\mathbb{P}_{1}$ is the space of polynomials of degree less
than or equal to 1. 

Further, let ${\bf \Phi}={\bf \Phi}(x)=(\phi_1(x),
\cdots, \phi_N(x))$ denotes a vector valued function consisting of
all the basis functions of $V_{h}$, where $N$ stands for the number of basis function (or degree of freedom). 
For any fixed $\omega \in \Gamma$, the vector ${\bf
u_{\epsilon}}(\omega)=(u_1(\omega),\cdots,u_N(\omega))^{T} \in \mathbb{R}^{N}$, then the FEM approximations of the state
variable $u_{\epsilon}$ and the control variable $f$ could be written as
\begin{eqnarray*}\label{FEMSpace}
\begin{aligned}
u_{h}(x,\omega)=\sum\limits_{j=1}^{N}u_{j}(\omega)\phi_{j}(x)\in U_{h},\quad
f_{h}(x)=\sum\limits_{j=1}^{N}f_{j}\phi_{j}(x)\in F_{h},
\end{aligned}
\end{eqnarray*}
where $\{\phi_{j}\}_{j=1}^{N}$ stands for basis functions of $U_h$,
and $N$ represents the degree of freedom on $D$.
Replacing $u_h$ and $f_h$ by ${\bf \Phi}{\bf
u_{\epsilon}}(\omega)$ and ${\bf \Phi}{\bf f}$, so (\ref{eq:approx weak form}) can be reformulated by 
\begin{equation}\label{eq:primal matrixvector}
A(\omega_i){\bf u_{\epsilon}}(\omega_i) =B_{1}{\bf f}, \quad i=1,\cdots,M,
\end{equation}
here, 
\begin{eqnarray}
\begin{aligned}
&\quad A(\omega_i)= (\alpha(x,\omega_i)\nabla\phi_{m},\nabla\phi_{n})_{L^{2}(D)}+\frac{1}{\epsilon}\left(\left(1-H_{\epsilon}(g)\right)\phi_{m},\phi_{n}\right)_{L^{2}(D)}\in \mathbb{R}^{N\times N},\\
\quad &\quad B_{1}=(\phi_{m},\phi_{n})_{L^{2}(D)}\in \mathbb{R}^{N\times N}.
\end{aligned}
\end{eqnarray}
Similarly, if the cost functional satisfies (\ref{eq:E obj}), then the matrix-vector form of (\ref{eq:dual weak form}) is 
\begin{equation}\label{eq:dual matrixvector}
A(\omega_i){\bf z}(\omega_i) =B_{2}{\bf (u_{\epsilon}-u_{d})}, \quad i=1,\cdots,M,
\end{equation}
with $B_{2}=(\left(1-H_{\epsilon}(g)\right)\phi_{m},\phi_{n})_{L^{2}(D)}\in \mathbb{R}^{N\times N},$ or equivalently, the right-hand side of (\ref{eq:dual matrixvector}) is $B_{3}(u_{\epsilon}-u_{d})$ with $B_{3}=\left(\phi_{m},\phi_{n}\right)_{L^{2}(O)}$.

\subsection{MC approximation and gradient descent method}
In this subsection, we mainly apply the standard MC method to
approximate the expectation operator $\mathbb{E}$. 
${\{\omega_{i}\}_{i=1}^{M}}$ ($M \gg 1$) are the identically distributed
independent samples chosen from the probability space
$(\Gamma,\mathcal{F},P)$. Because of the law of large numbers and the
central limit theorem, for any state variable $u\in
L^{2}(\Gamma;H^{1}(D))$, we can approximate the expectation
$\mathbb{E}[u_h(x,\omega)]$  by the average of the $M$ samples
\begin{equation}\label{eq:appro expec oper}
   \mathbb{E}_{M}[u_h(x,\omega)]=\frac{1}{M}\sum_{i=1}^{M} u_h(x,\omega_{i}),
\end{equation}
After FEM and MC, we can derive the following discretized shape optimization objective functional
\begin{equation}
 \min\limits_{g}E_{M,h}\left[\int_{\Omega}J(x,\omega,u_{h}(x,\omega))dx\right]
\end{equation} 
In the following, an accelerated version of gradient descent method is used to solve the above optimal problem and update the parameter $g$. Now, we can describe the update process in the following algorithm.
\begin{algorithm}[H]\caption{FEM-MC-AGD}\label{al:FEM-MC-GD}
\renewcommand{\algorithmicrequire}{\textbf{Input:}}
\begin{algorithmic}[1]
\Require 
 Initial shape function  ${\bf{g}}^{0}$, positive step sizes $\alpha_{0}$, random sample set $\{\omega_{i}\}_{i\geq 1}$, force term $f$, desired state $u_{d}$ and regularization parameter $\epsilon$, $k=0$;
 \Ensure the optimal shape related to ${\bf{g}}^{k+1}$;
 \State Generate constant  mass matrix $B_{1}$;
\While {not stopping condition}
\ForAll{$i=1,2,\cdots, M$}
\State Generate the stochastic approximate stiffness matrix $A(\omega_{i})$, compute the solution vector $ {\bf{u}}^{k}_{\epsilon}(\omega_{i})$ of  primal equation by (\ref{eq:primal matrixvector});
\State Compute the cost value $j({\bf{g}^{k+1}})$;
\State Generate the approximate mass matrix $B_{2}$, compute the solution vector $ {\bf{z}}^{k}(\omega_{i})$ of  dual equation (if (\ref{eq:E obj}) holds, then by (\ref{eq:dual matrixvector}));
\State Compute the stochastic gradient vector $\nabla  j({\bf g}^k)$ by using $ {\bf{u}}^{k}_{\epsilon}(\omega_{i})$, ${\bf{z}}^{k}(\omega_{i})$;

\EndFor
\State Compute the average gradient $\nabla \bar  j({\bf g}^k)$  and the average  cost $\bar{j}({\bf g}^{k+1})$; 
\State Update ${\bf g}^{k+1}={\bf g}^{k}-\alpha_{k}\nabla \bar  j({\bf g}^k)$, where $\alpha_{k}$ satisfy Armijo condition;
\State  $k=k+1$;
\EndWhile
 
 \end{algorithmic}
\end{algorithm}
\begin{remark}
In the above algorithm, the stopping condition  include setting the maximum iteration steps, $k=1,2,\cdots K_{max}$, using the relative (absolute) error between ${\bf g}^k,{\bf g}^{k+1}$ or $j({\bf g}^k),j({\bf g}^{k+1})$. For example, we stop the algorithm if $|j({\bf g}^{k+1})-j({\bf g}^{k})|<1e-8$ or $\|{\bf g}^{k+1}-{\bf g}^{k}\|_{2}<1e-8$ holds, where $\frac{\|{\bf g}^{k+1}-{\bf g}^{k}\|_{2}}{\|{\bf g}^{k+1}\|_{2}}$ is used instead of $\|{\bf g}^{k+1}-{\bf g}^{k}\|_{2}$ if $\|{\bf g}^{k+1}\|_{2}\geq 1$, and analogous for $j({\bf g}^{k})$ using absolute value to measure.
\end{remark}
\begin{remark}
When considering  objective functional (\ref{eq:O obj}) and $O\subset K\subset D$, we need to project the solution into the domain $K$  because constraint $g\geq 0$ has been imposed on this set.
\end{remark}
\begin{remark}
An enhanced search scheme---three-point quadratic interpolation  coupled with quadratic fitting method is applied to search for the optimal step size in gradient descent method. More specifically,  it evaluates cost at three strategically chosen points relative to the historical step size:
$[0.5\alpha_{prev}, \alpha_{prev}, 1.5\alpha_{prev}] $,  then fits a quadratic polynomial to the cost values and computes the exact minimum. We also set the minimum, maximum step size $\alpha_{min}=1,\alpha_{max}=10$,  separately.  In the iteration process, momentum gradient smoothing is also used to accelerate convergence, which combines current gradient with historical information.
\end{remark}
\section{Numerical experiments}\label{sec4}
The purpose of this  section is to illustrate the validity and  main features from the previous sections 
with the propose four numerical examples.

We consider the optimization shape problem on the spatial domain
$D=(-1,1)^2\subset \mathbb{R}^{2}$, the diffusion coefficient $\alpha(x,\omega)$
is set to be
\begin{equation}\label{eq:coeffa}
\alpha(x,\omega)=1+\rho\eta(x,\omega),
\end{equation}
other cases could be considered similarly. Here,  $\rho=0.01$ stands for the disturbing degree, $\eta(x,\omega)\in
L^{2}(\Gamma;L^{\infty}(D))$ denotes the random process satisfying
\begin{equation*}\label{eq:eta}
\mathcal{P}{\{\omega \in \Gamma;\|\eta(x,\omega)\|_{L^{\infty}(D)}\leq
1\}=1},
\end{equation*}
and  obeys the
uniform distribution on the interval of $[-1,1]$ for any fixed $x\in
D$. 
Fixed regularization parameter $\epsilon=1e-5$, the maximum iteration steps $K_{max}=1000$ and sample number $M=100$. Considering the samples are independent and identify distribution and the structure of the algorithms, we adopt a parallel strategy to  accelerate the running speed. In addition, we observe $H_{\epsilon}(g)$ is a monotone increasing function, so the simplified gradient which drop the term $H'_{\epsilon}(g)$ can be  applied to the algorithm, this will  reduce the computation. The finite element mesh size $h=\frac{1}{128}$, including $32258$ triangular elements and 16384 vertices.
\begin{example}\label{ex1}
In this example, we consider the objective functional 
\begin{equation*}
j(x,\omega,u(x,\omega))=\frac{1}{2}\mathbb{E}\left[\int_{O}(u-u_{d})^2dx\right]
\end{equation*}
with $u_{d}(x)=-\left(x_{1}-\frac{1}{2}\right)^2-\left(x_{2}-\frac{1}{2}\right)^2+\frac{1}{8}$ and the fixed square domain $O=[-\frac{1}{2},\frac{1}{2}]^2\subset D$. The initial shape function is 
\begin{equation*} g^{0}(x)=\frac{7}{8}-x_{1}^2-x_{2}^{2},
\end{equation*} 
the right-hand side function $f=2$.
Besides, we impose the extra constraint $g\geq 0$ in $O$. In the algorithm,  the simplified descent direction is the negative gradient direction
$d=-\frac{1}{\epsilon}\mathbb{E}\left[u_{\epsilon}z\right]$.  
\end{example} 
Figure~ \ref{fig:ex1initialshape} shows the initial geometry with the boundary, where the solid blue line represents the boundary of $g^0(x)$, the black dotted line represents area $D$ and the dotted line in magenta indicates the boundary of area $O$, while the optimal geometry is none other than the square domain $O$. 
 \begin{figure}[H]
 \begin{center}
 \includegraphics[width=0.9\textwidth]{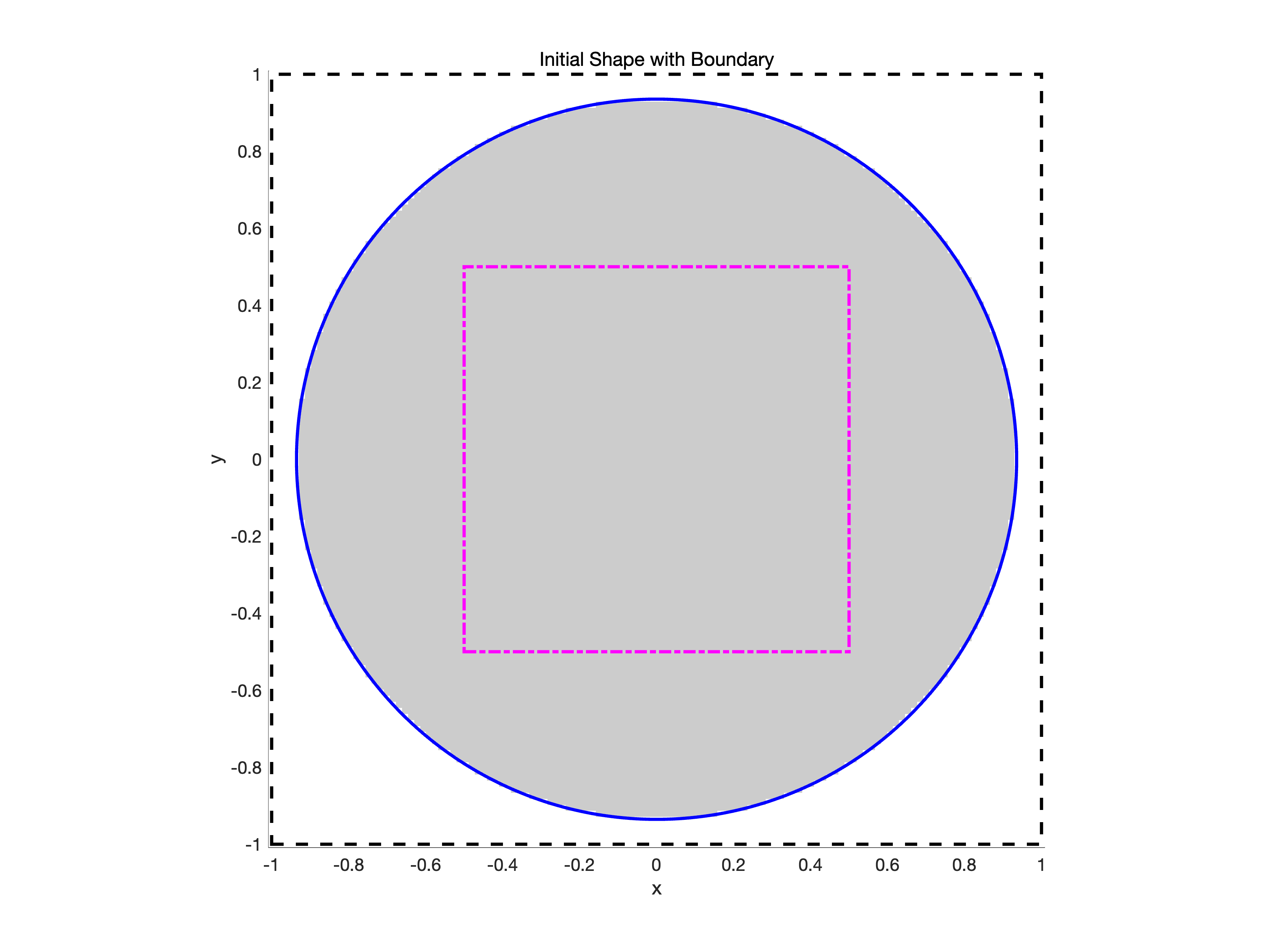}
 \caption{The initial shape with boundary.}
 \label{fig:ex1initialshape}
 \end{center}
 \end{figure}
 Figure~\ref{fig:ex1mediumshape} indicates  some intermediate shape during iteration, the corresponding expected cost values are $0.475400$, $0.269463$, $0.265508$, $0.265508$, $0.265507$, $0.265507$. In the last step, the absolute change of the costs $|j({\bf g}^{k+1})-j({\bf g}^{k})|$ is  $2.1103e-10$, and the corresponding expected cost is $0.265507$. We also present the iterative variation curve of the expected cost in Figure~\ref{fig:ex1expectedcost}.
 \begin{figure}[H]
 \begin{center}
 \begin{minipage}{0.325\linewidth}
 \includegraphics[width=0.9\textwidth]{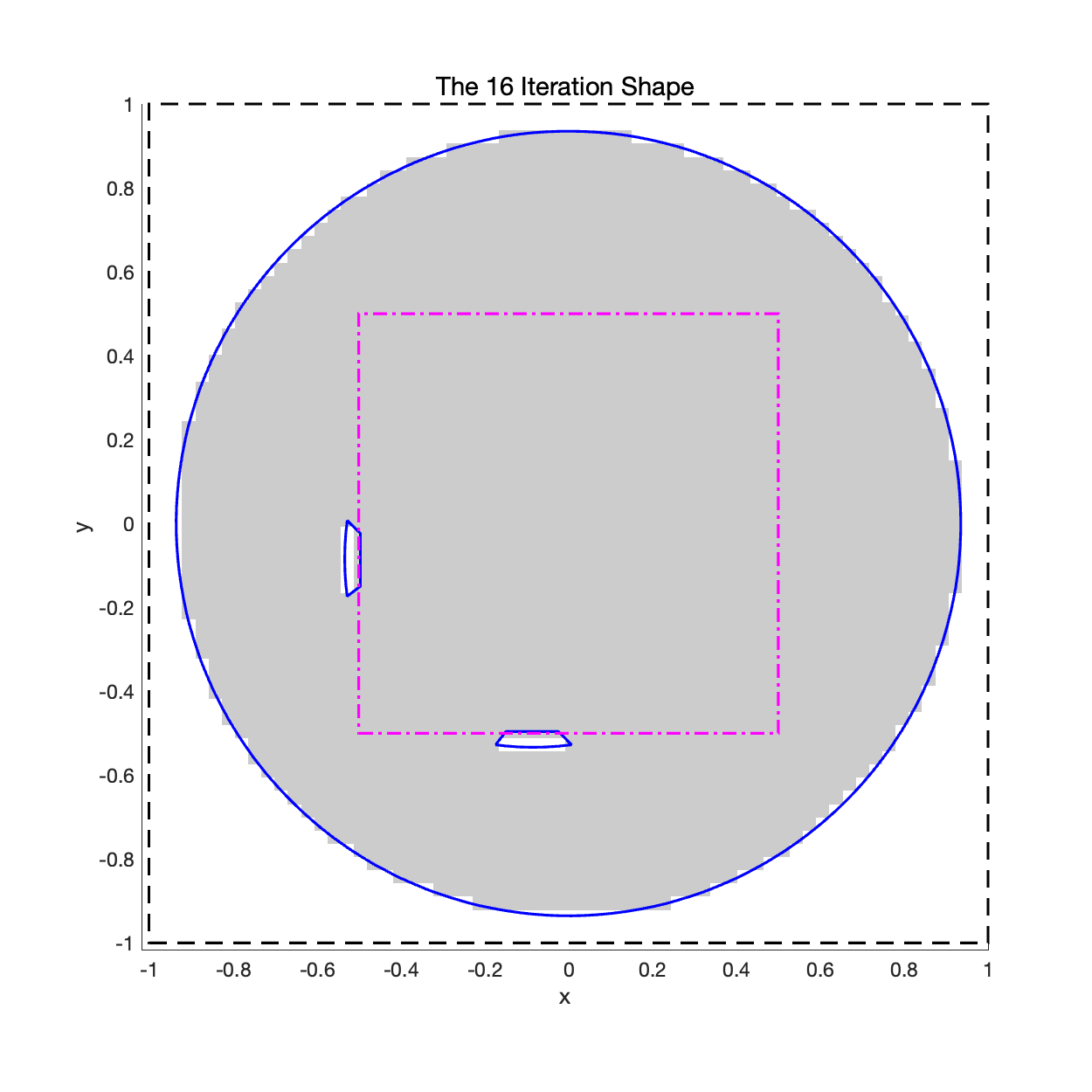}
 \end{minipage}
 \begin{minipage}{0.325\linewidth}
 \includegraphics[width=0.9\textwidth]{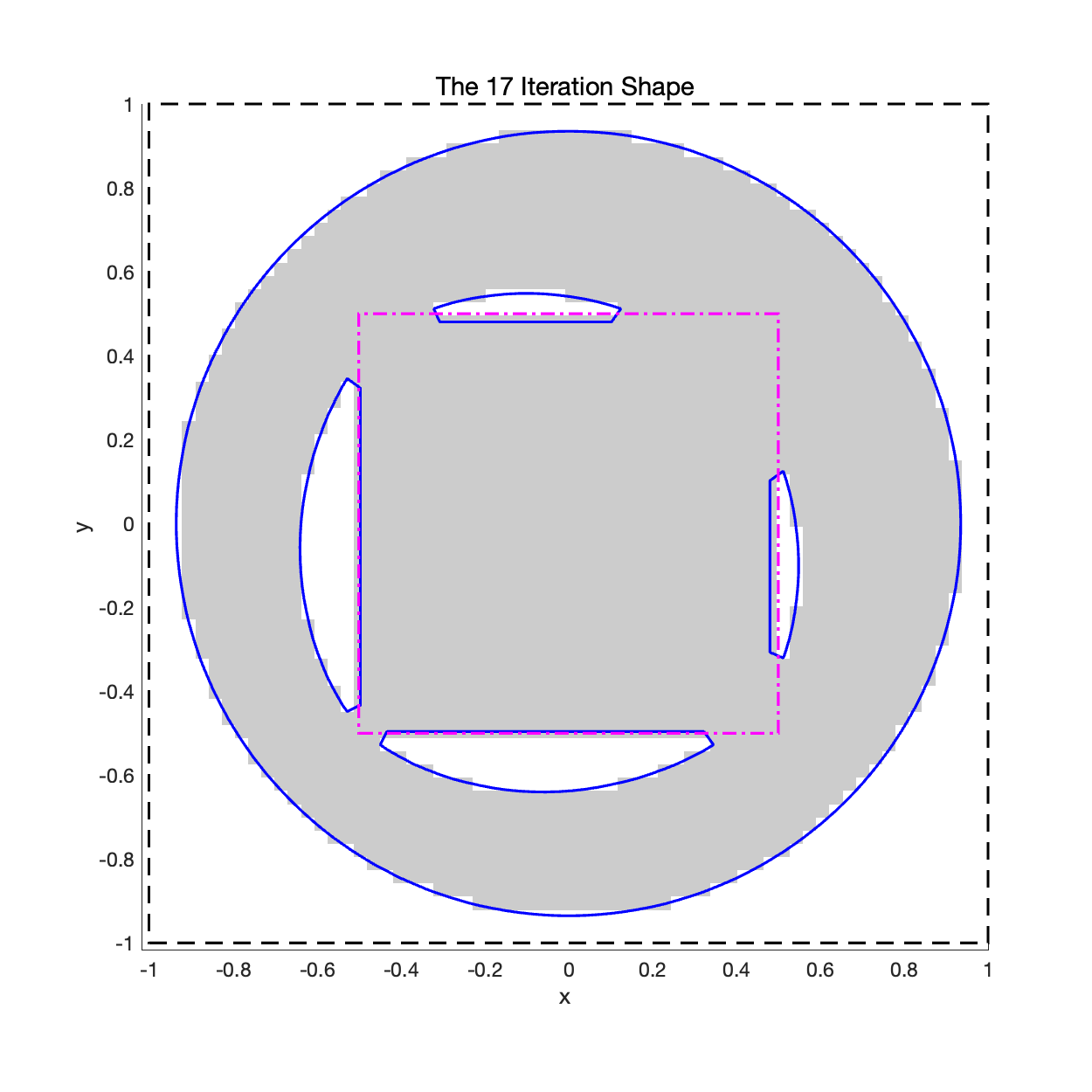}
 \end{minipage}
 \begin{minipage}{0.325\linewidth}
 \includegraphics[width=0.9\textwidth]{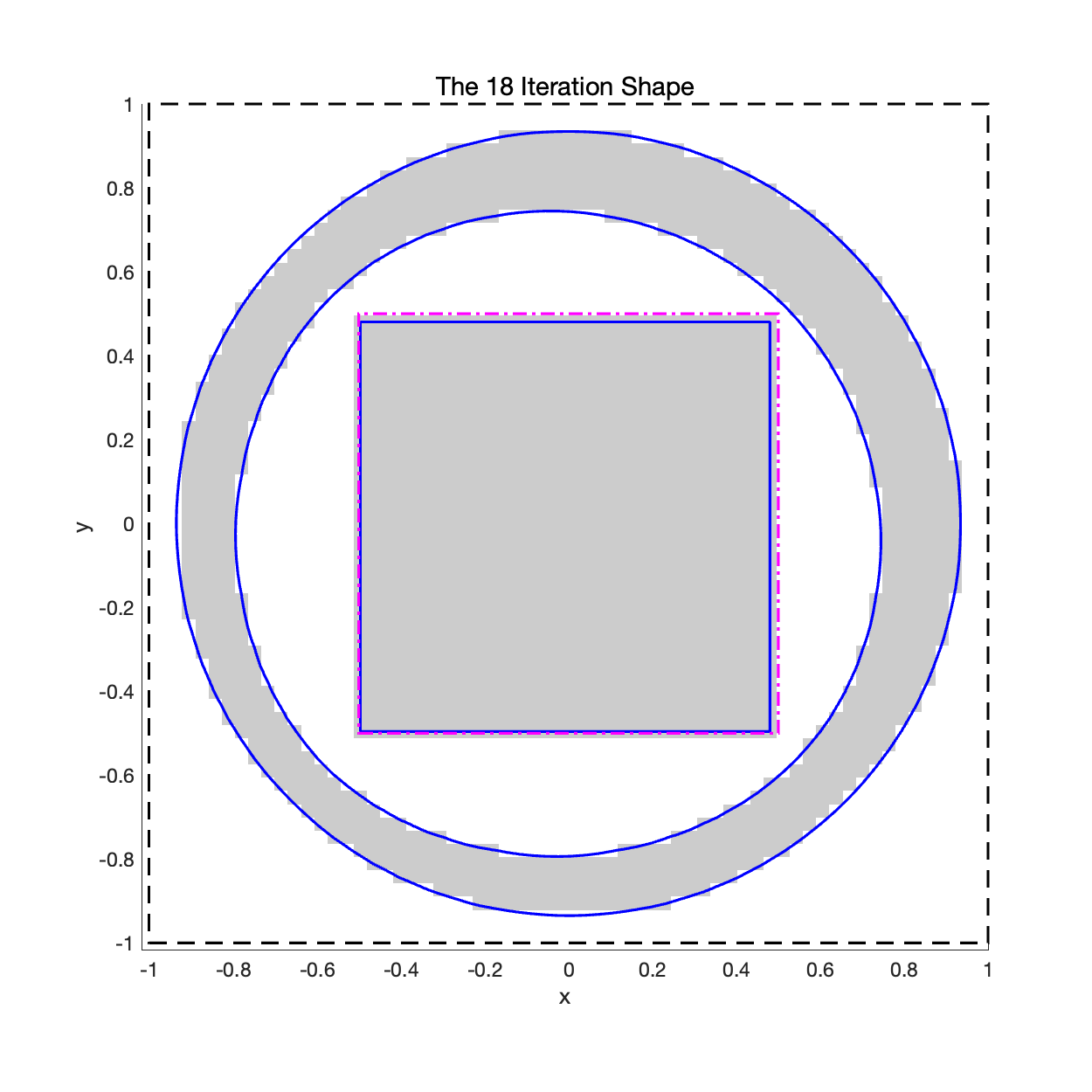}
 \end{minipage}
 \qquad
 \begin{minipage}{0.325\linewidth}
 \includegraphics[width=0.9\textwidth]{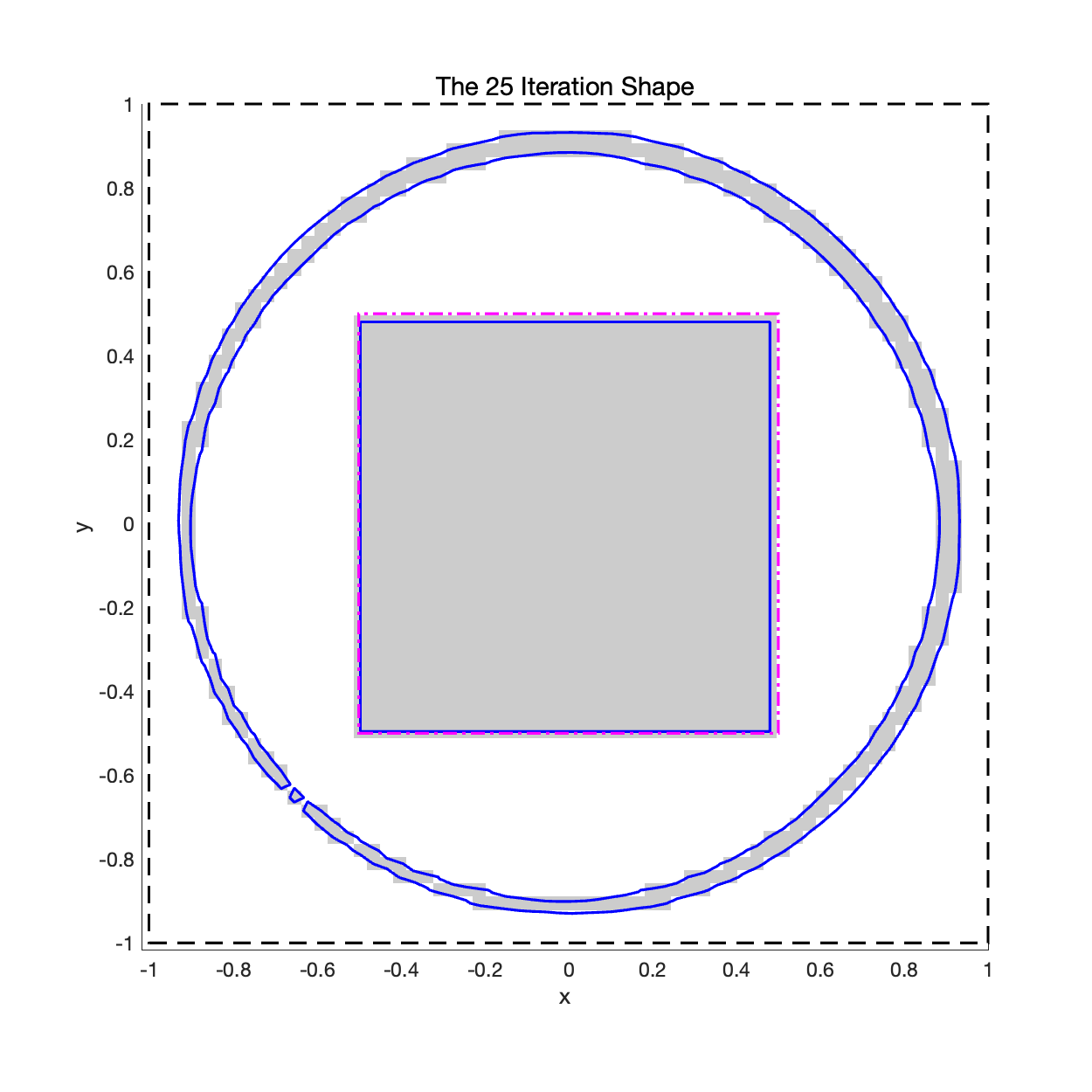}
 \end{minipage}
 \begin{minipage}{0.325\linewidth}
 \includegraphics[width=0.9\textwidth]{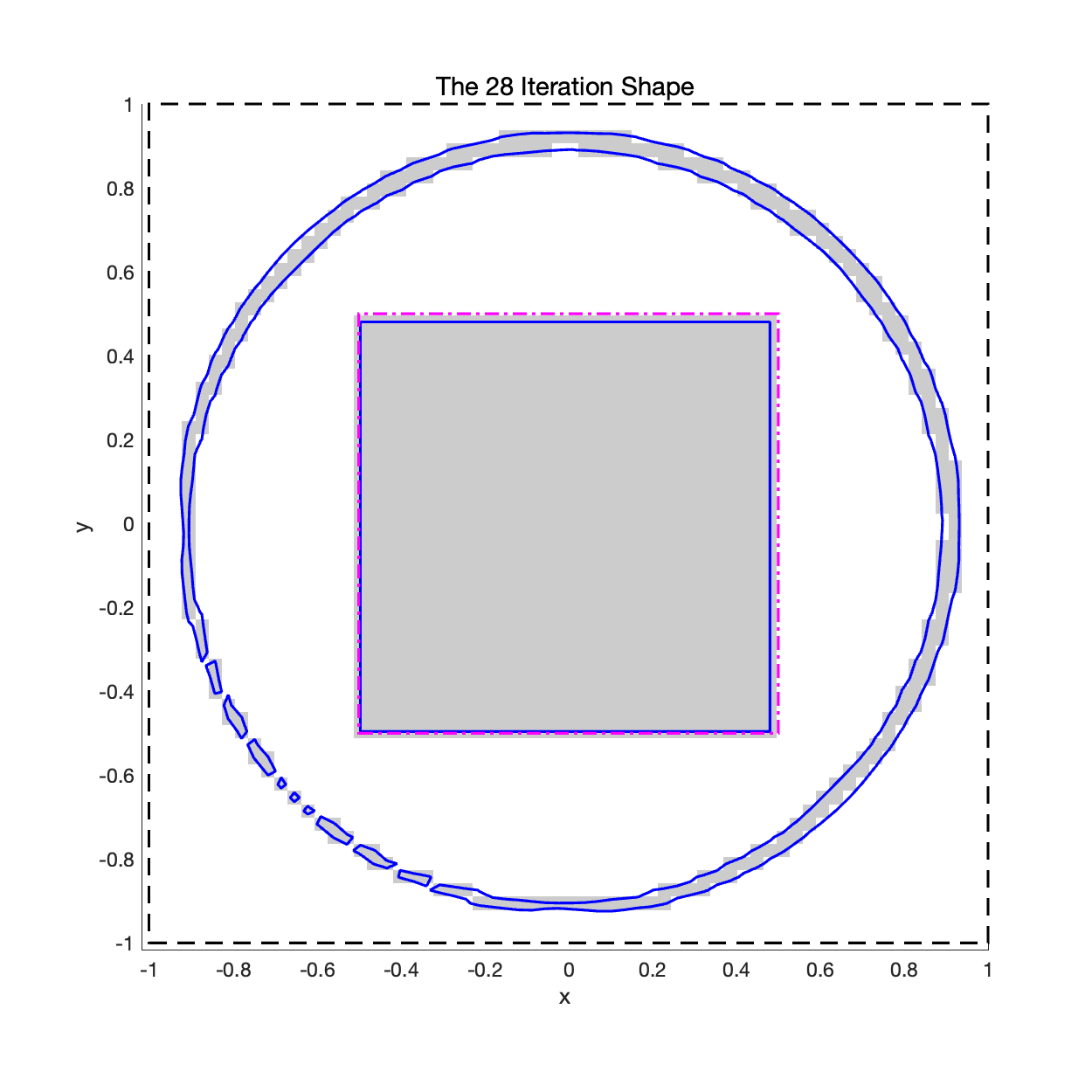}
 \end{minipage}
 \begin{minipage}{0.325\linewidth}
 \includegraphics[width=0.9\textwidth]{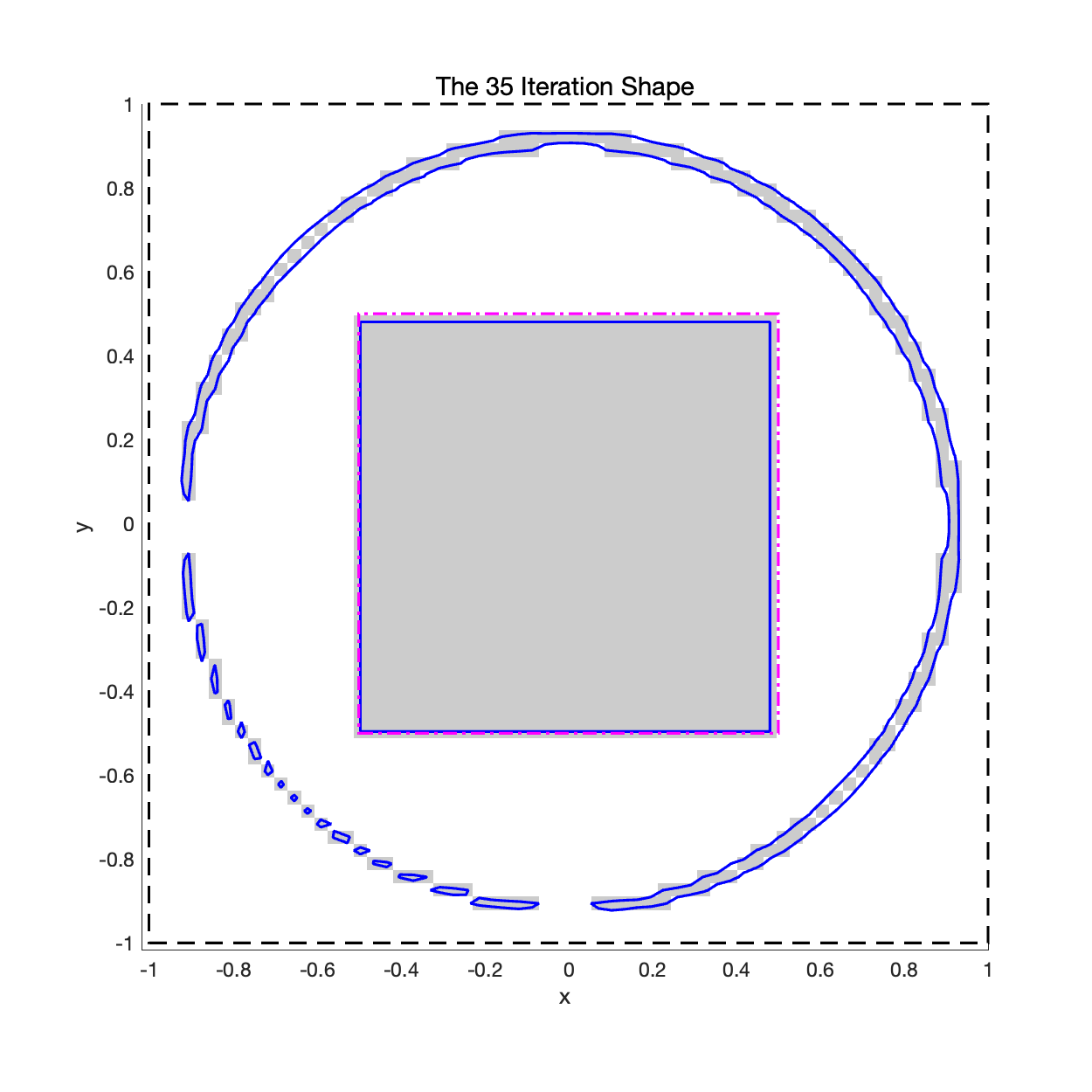}
 \end{minipage}
 \caption{The iterative shapes with boundary of steps 16, 17, 18, 25, 28 and 35.}
 \label{fig:ex1mediumshape}
 \end{center}
 \end{figure}

 \begin{figure}[H]
 \begin{center}
 \includegraphics[width=0.9\textwidth]{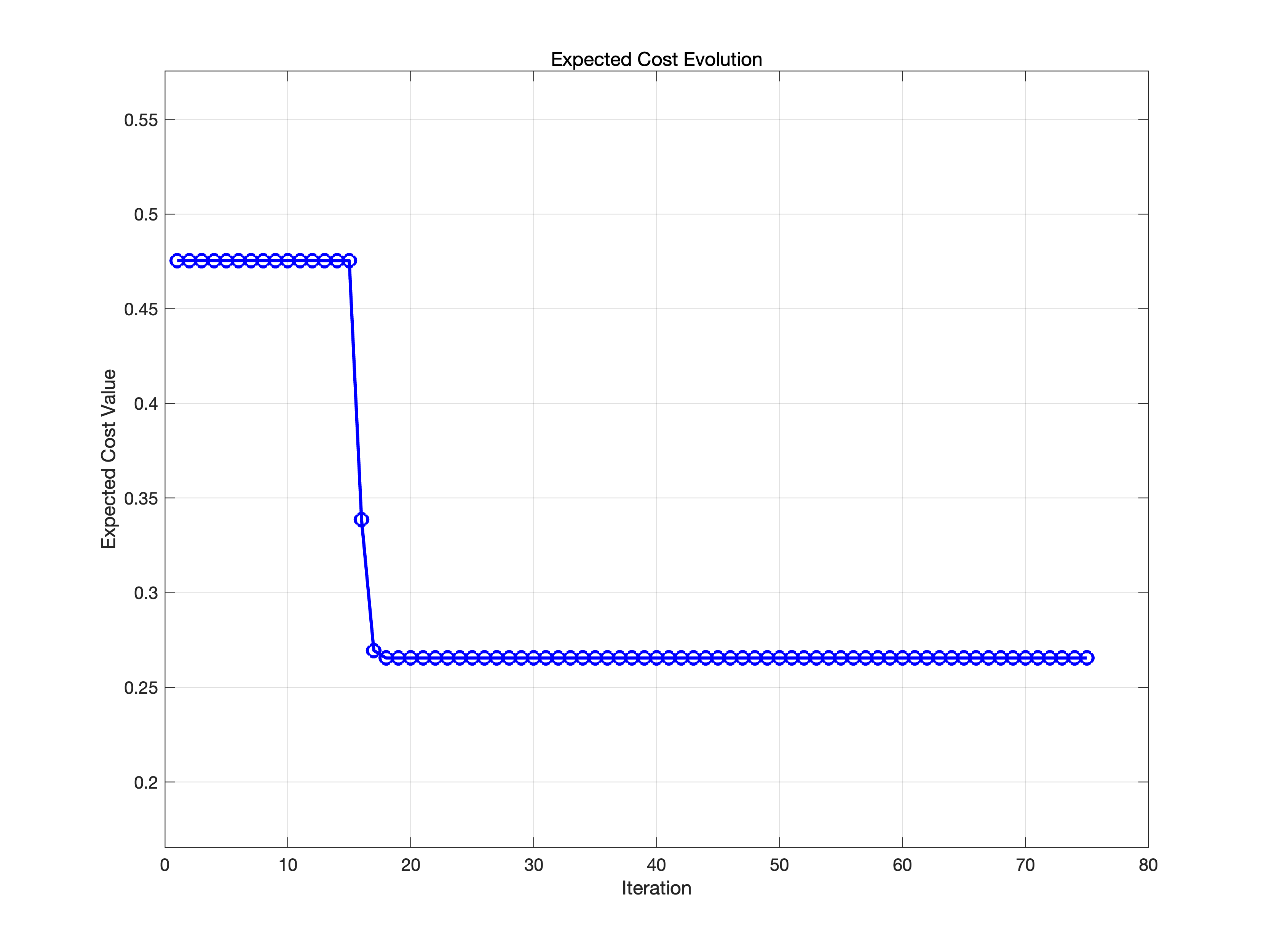}
 \caption{The variations of expected cost cure with iterations of Example~\ref{ex1}.}
 \label{fig:ex1expectedcost}
 \end{center}
 \end{figure}
 
 \begin{example}\label{ex2}
 Based on Example~\ref{ex1},  and now  let  the force term is a function depend on the spatial coordinate, that is $f(x)=2\pi^{2}sin(\pi x_1)sin(\pi x_2)+1$, and $u_d=sin(\pi x_1)sin(\pi x_2)$. 
\end{example}
We show some intermediate geometries  for shape optimizations using the same initial shape as Example~\ref{ex1}.  We observe that the intermediate iteration shape and the final shape are different from the previous example.  Hence, the square area $O$ is the theoretical optimal shape. We also notice that the shape update starts from the upper right corner and the lower left corner, it has  central symmetry, these symmetries are visible in  Figure~\ref{fig:mediumshape2}.  Figure~\ref{fig:ex2expectedcost} shows the expected cost evolution, in this example, the cost  value at the final step of the iteration  is $0.163355$, the  absolute change of cost value $|j({\bf{g}}^{k+1})-j({\bf{g}}^k)|$ is $9.6571e-10$ and the relative change  $\frac{\|{\bf{g}}^{k+1}-{\bf{g}}^{k}\|_{2}}{\|{\bf{g}}^{k+1}\|_{2}}$ is $1.3943e-04$.
\begin{figure}[H]
\begin{center}
\begin{minipage}{0.325\linewidth}
\includegraphics[width=0.9\textwidth]{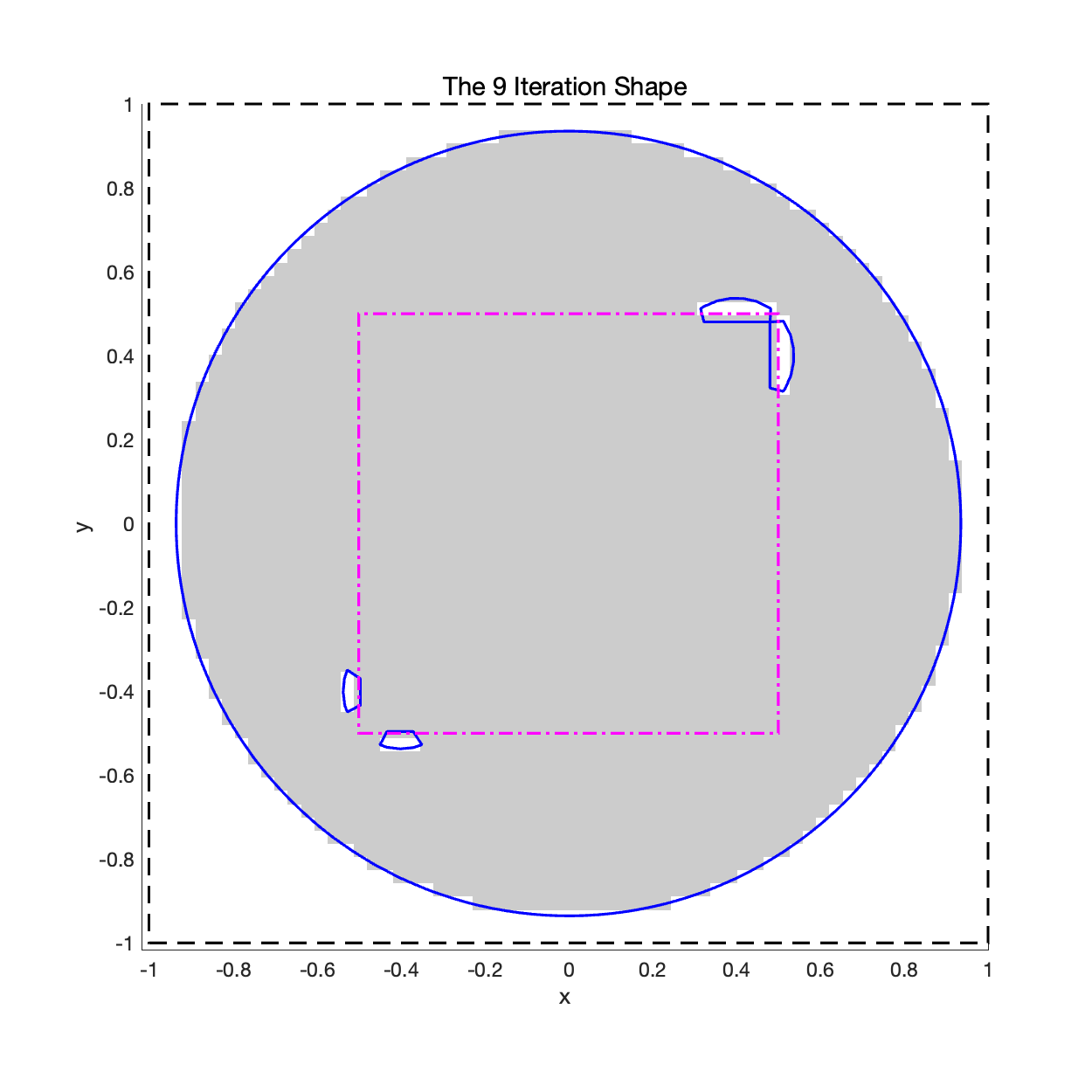}
\end{minipage}
\begin{minipage}{0.325\linewidth}
\includegraphics[width=0.9\textwidth]{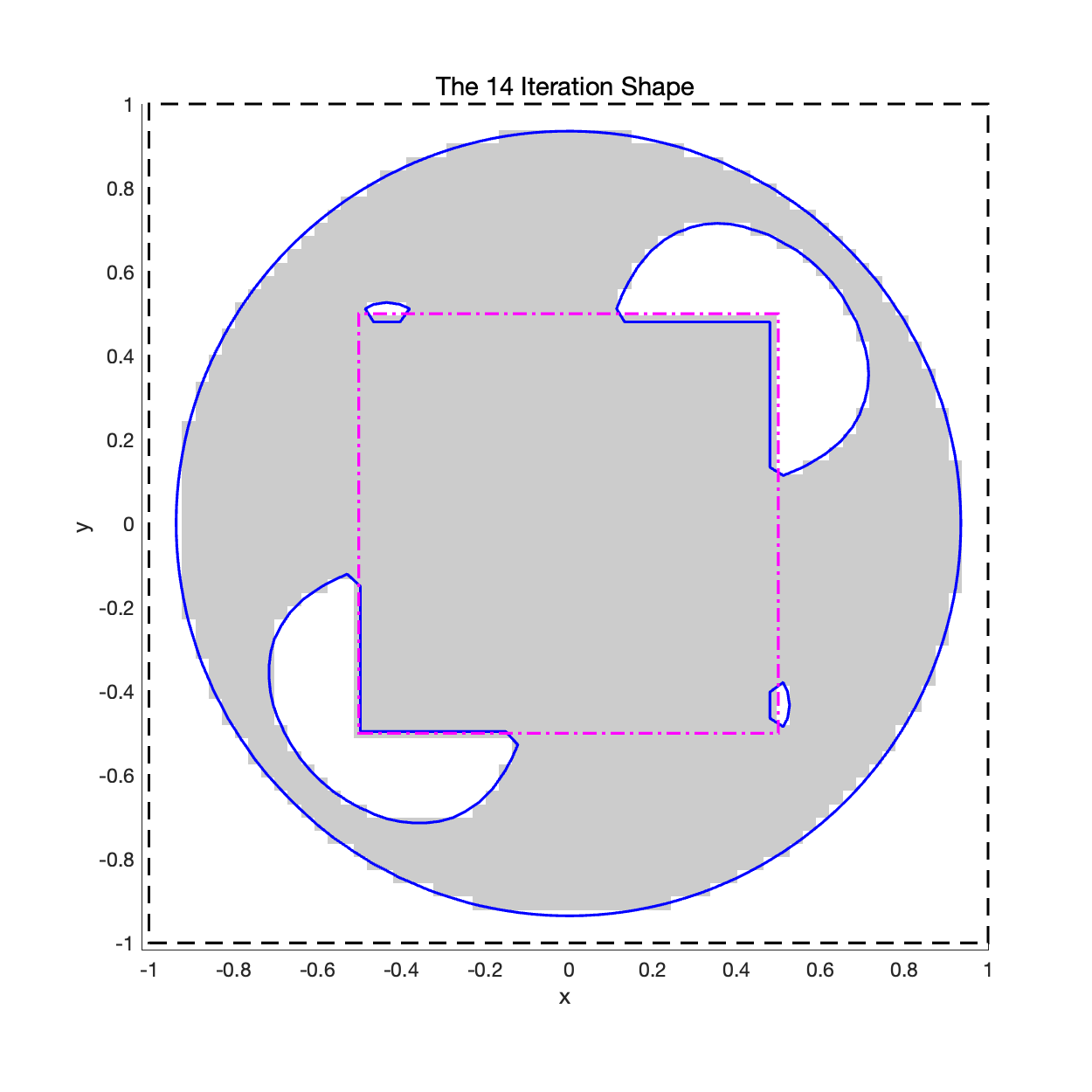}
\end{minipage}
\begin{minipage}{0.325\linewidth}
\includegraphics[width=0.9\textwidth]{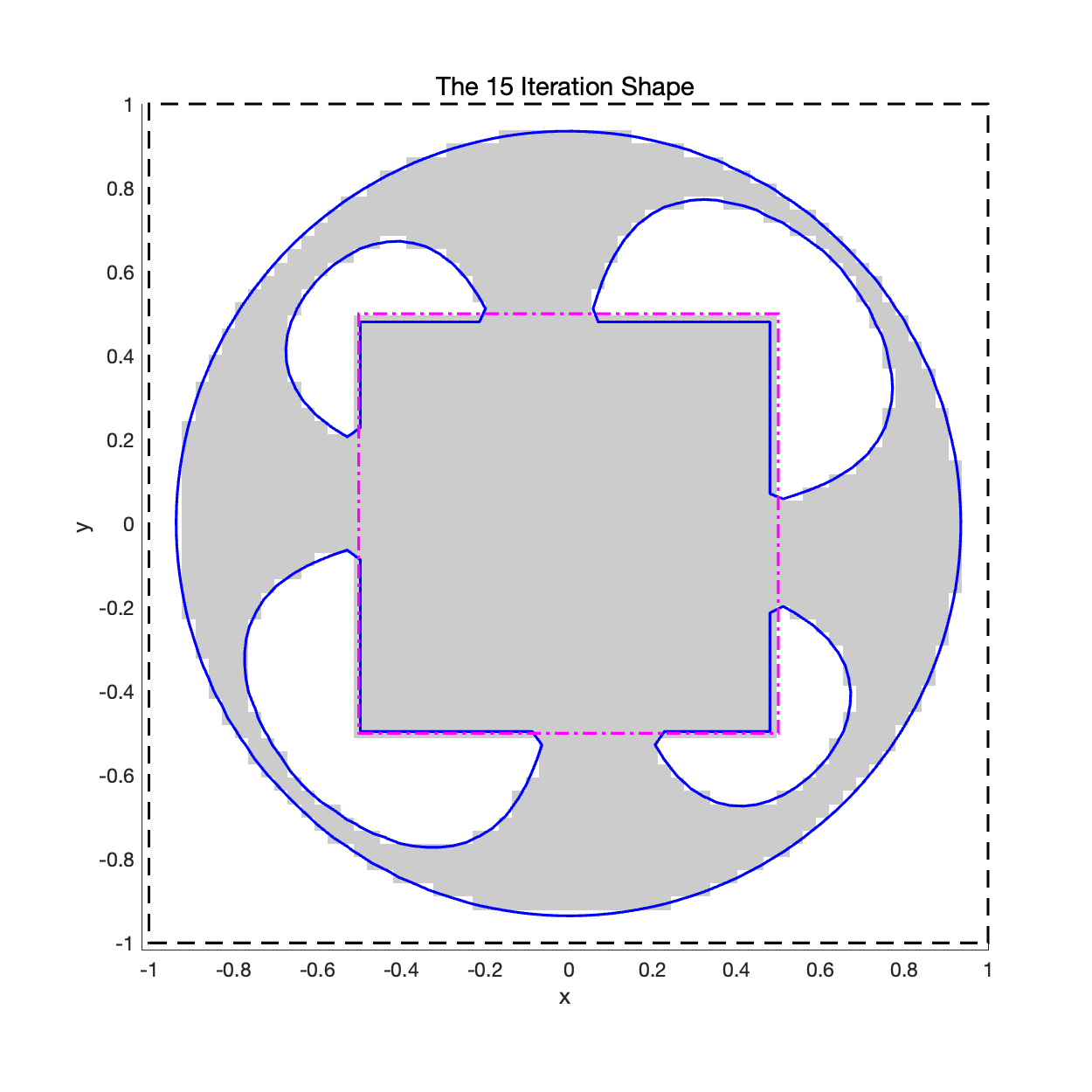}
\end{minipage}
\qquad
\begin{minipage}{0.325\linewidth}
\includegraphics[width=0.9\textwidth]{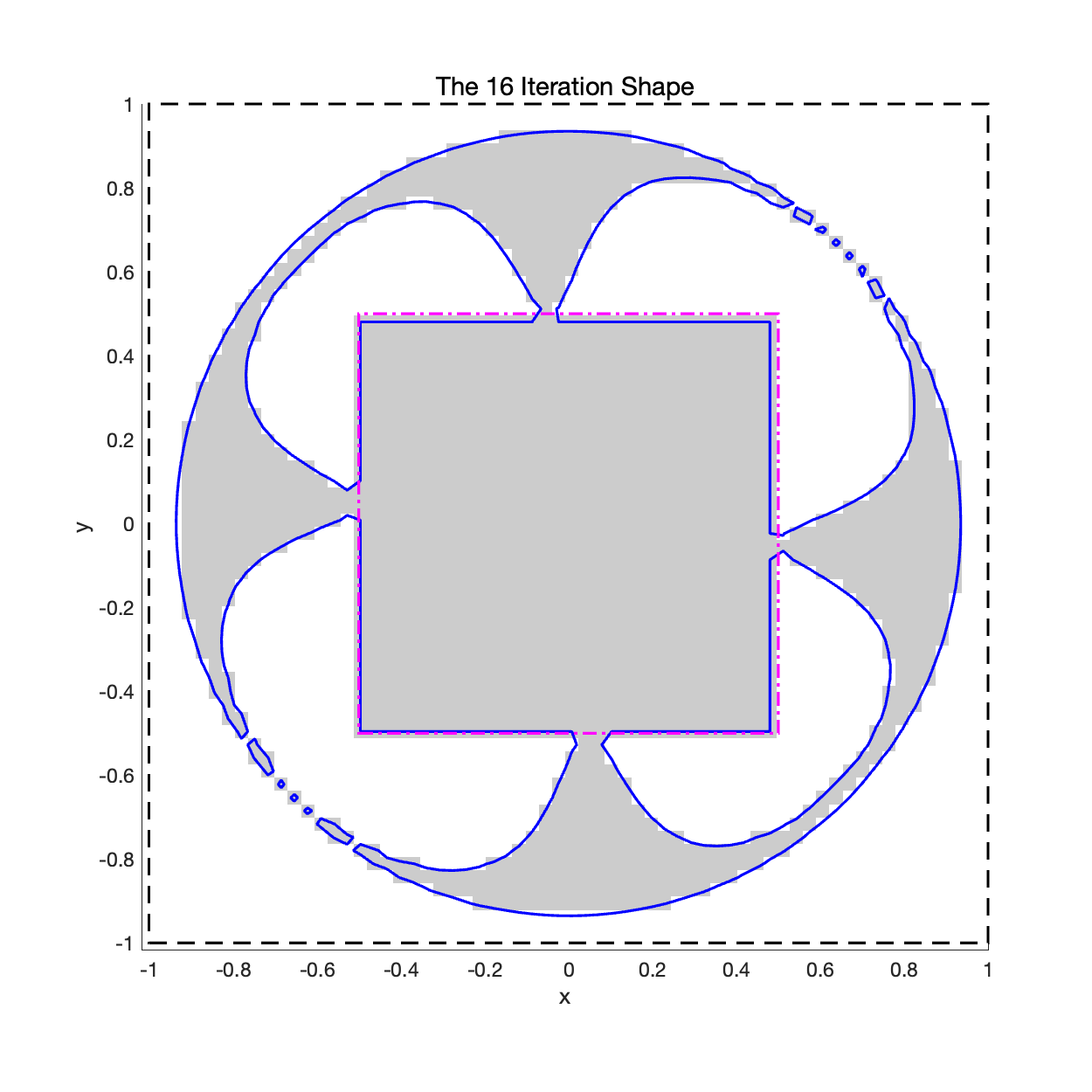}
\end{minipage}
\begin{minipage}{0.325\linewidth}
\includegraphics[width=0.9\textwidth]{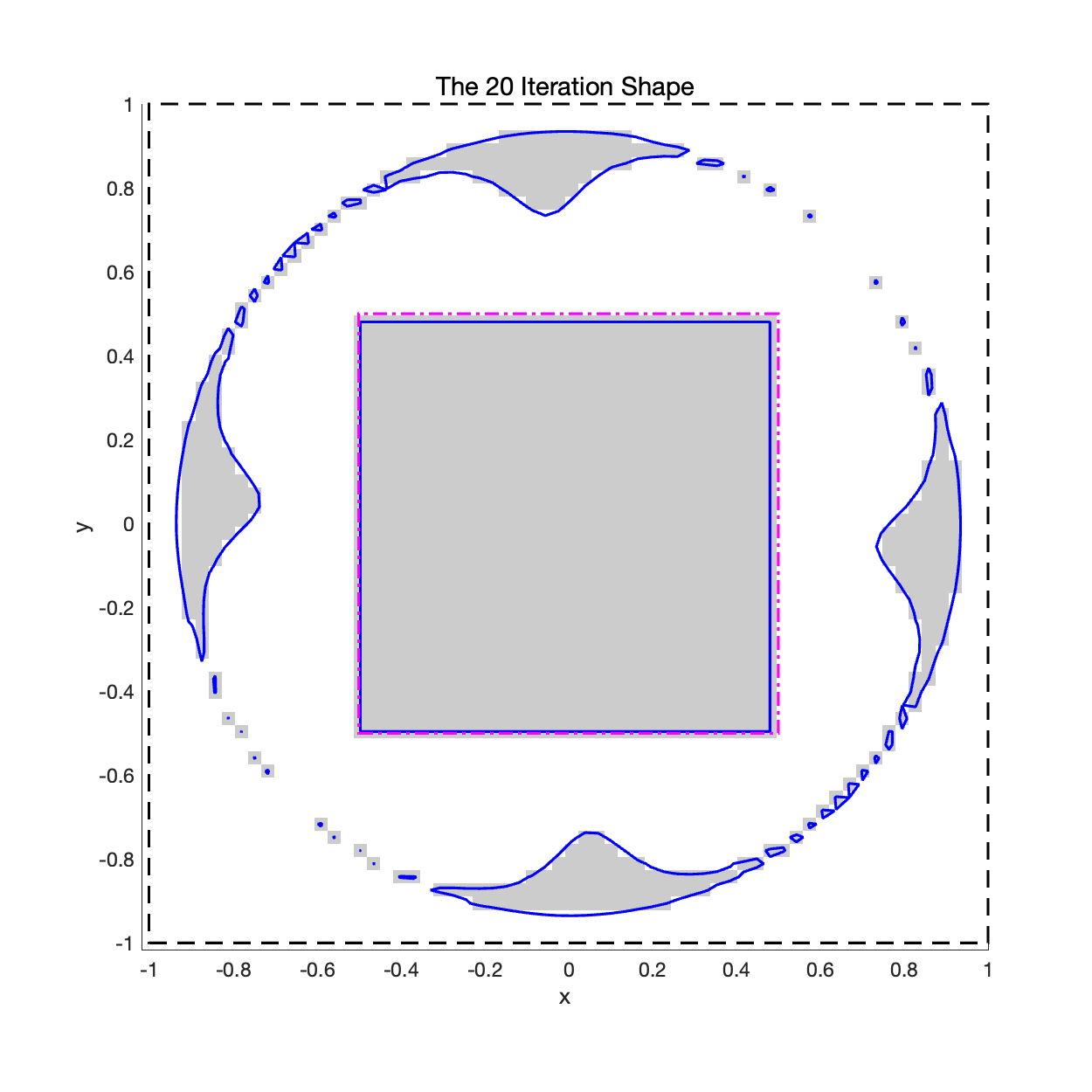}
\end{minipage}
\begin{minipage}{0.325\linewidth}
\includegraphics[width=0.9\textwidth]{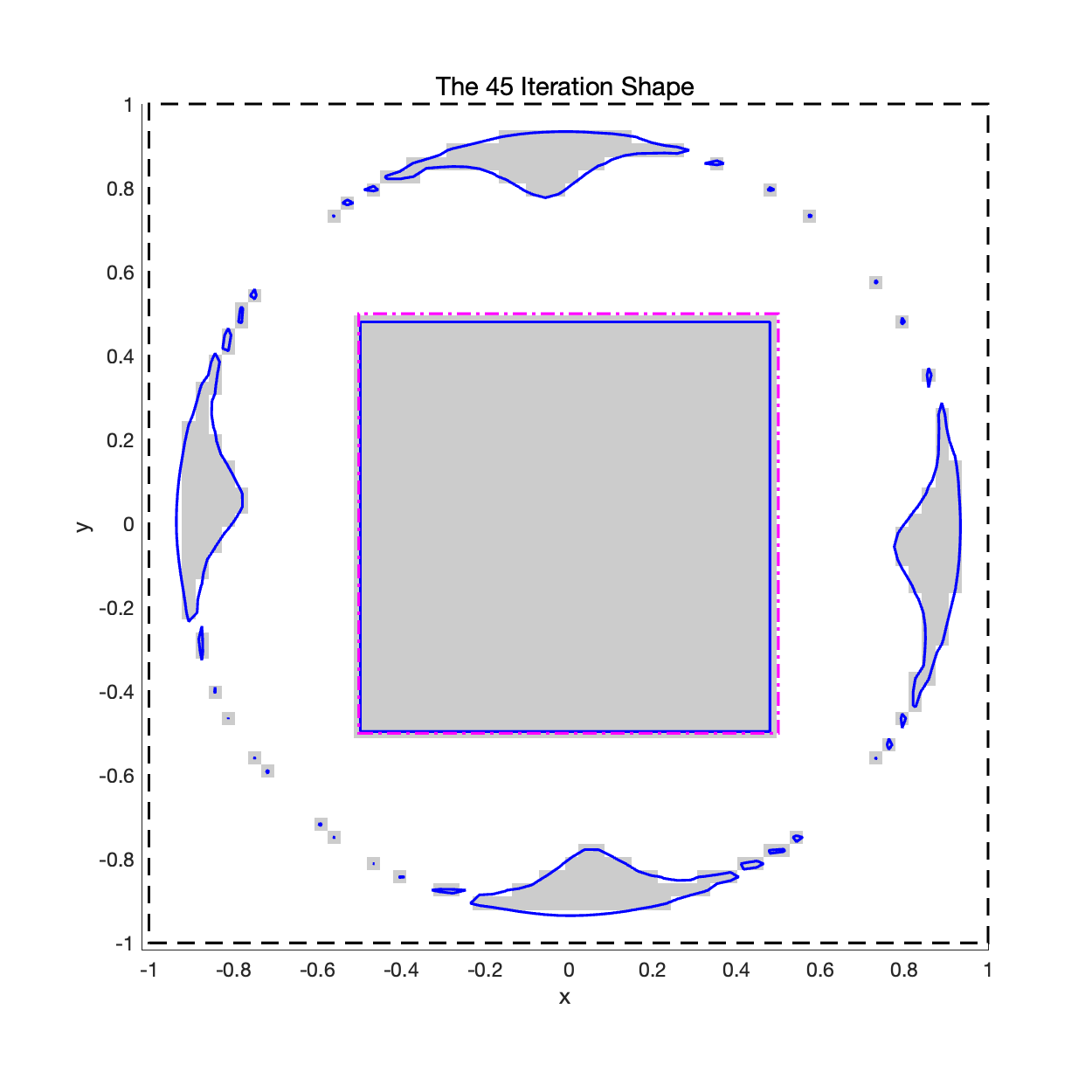}
\end{minipage}
\caption{The iterative shapes with boundary of steps 9, 14, 15, 16,  20, and 45.}
\label{fig:mediumshape2}
\end{center}
\end{figure}

\begin{figure}[H]
 \begin{center}
 \includegraphics[width=0.9\textwidth]{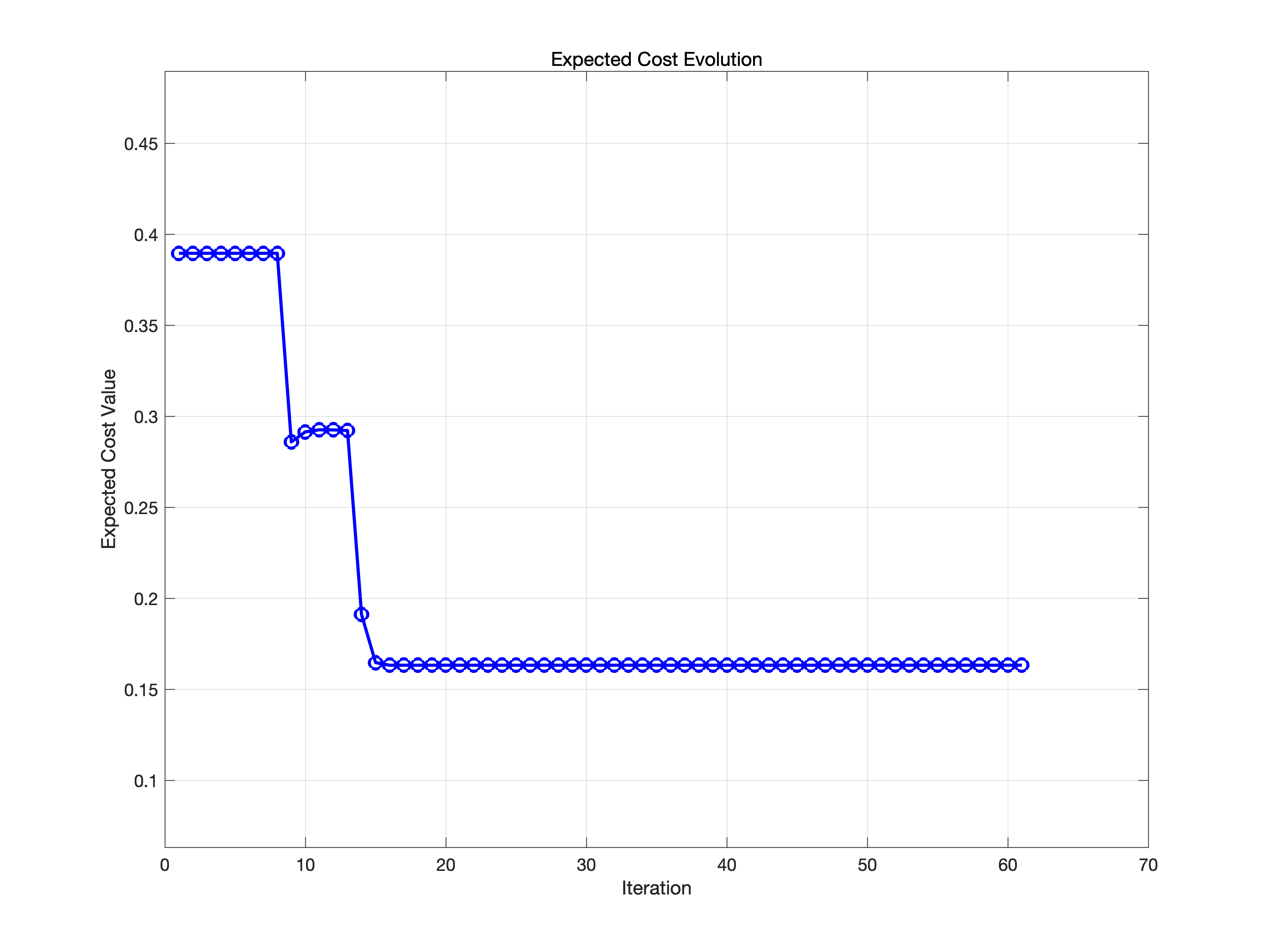}
 \caption{The variations of expected cost cure with iterations of Example \ref{ex1}.}
 \label{fig:ex2expectedcost}
 \end{center}
 \end{figure}
 
\begin{example}\label{ex3}
In this example, we only modify the subdomain $O$ to a disk area with the center at $(0,0)$ and a radius of $\frac{1}{2}$, other parameter settings are  same with Example~\ref{ex2}.
\end{example}
 The graphical representation about expected cost evolution is shown in Figure~\ref{fig:mediumshape3}. Figure~\ref{fig:ex3expectedcost} shows the expected cost evolution, In this example, the variation accuracy order of cost absolute value $|j({\bf{g}}^{k+1})-j({\bf{g}}^k)|$  can also achieve $1e-10$,  and the final expected cost is $0.065557$. 

\begin{figure}[H]
\begin{center}
\begin{minipage}{0.325\linewidth}
\includegraphics[width=0.9\textwidth]{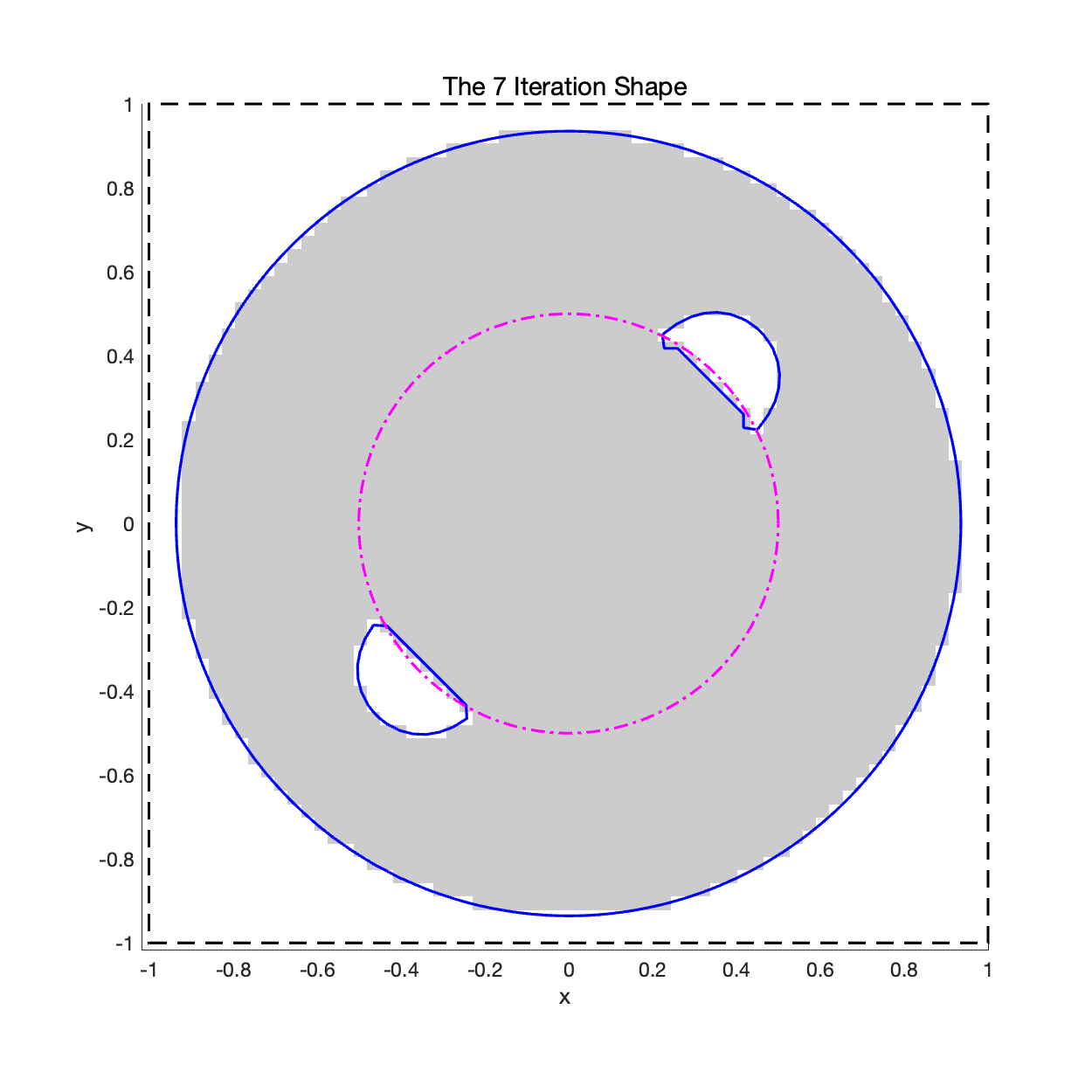}
\end{minipage}
\begin{minipage}{0.325\linewidth}
\includegraphics[width=0.9\textwidth]{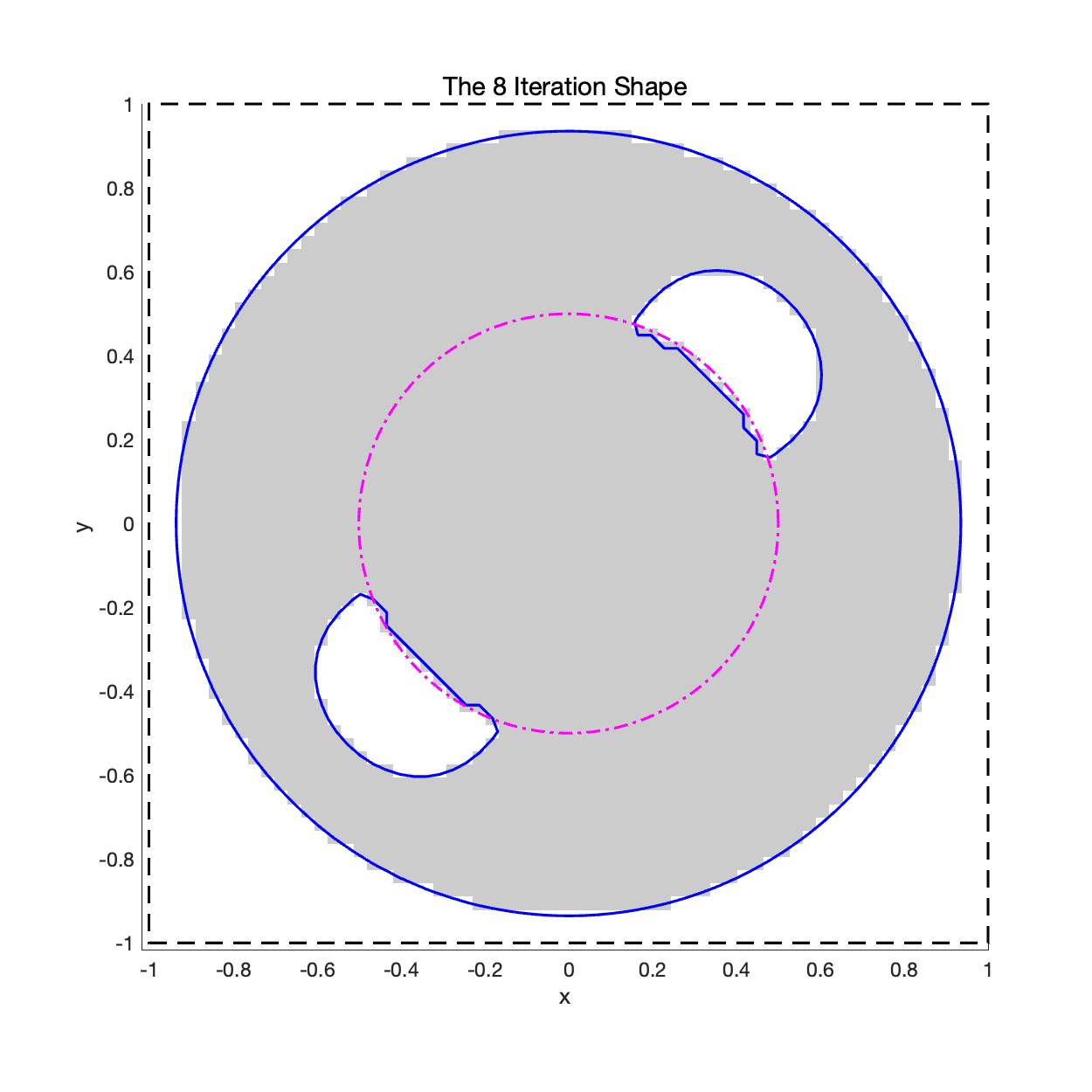}
\end{minipage}
\begin{minipage}{0.325\linewidth}
\includegraphics[width=0.9\textwidth]{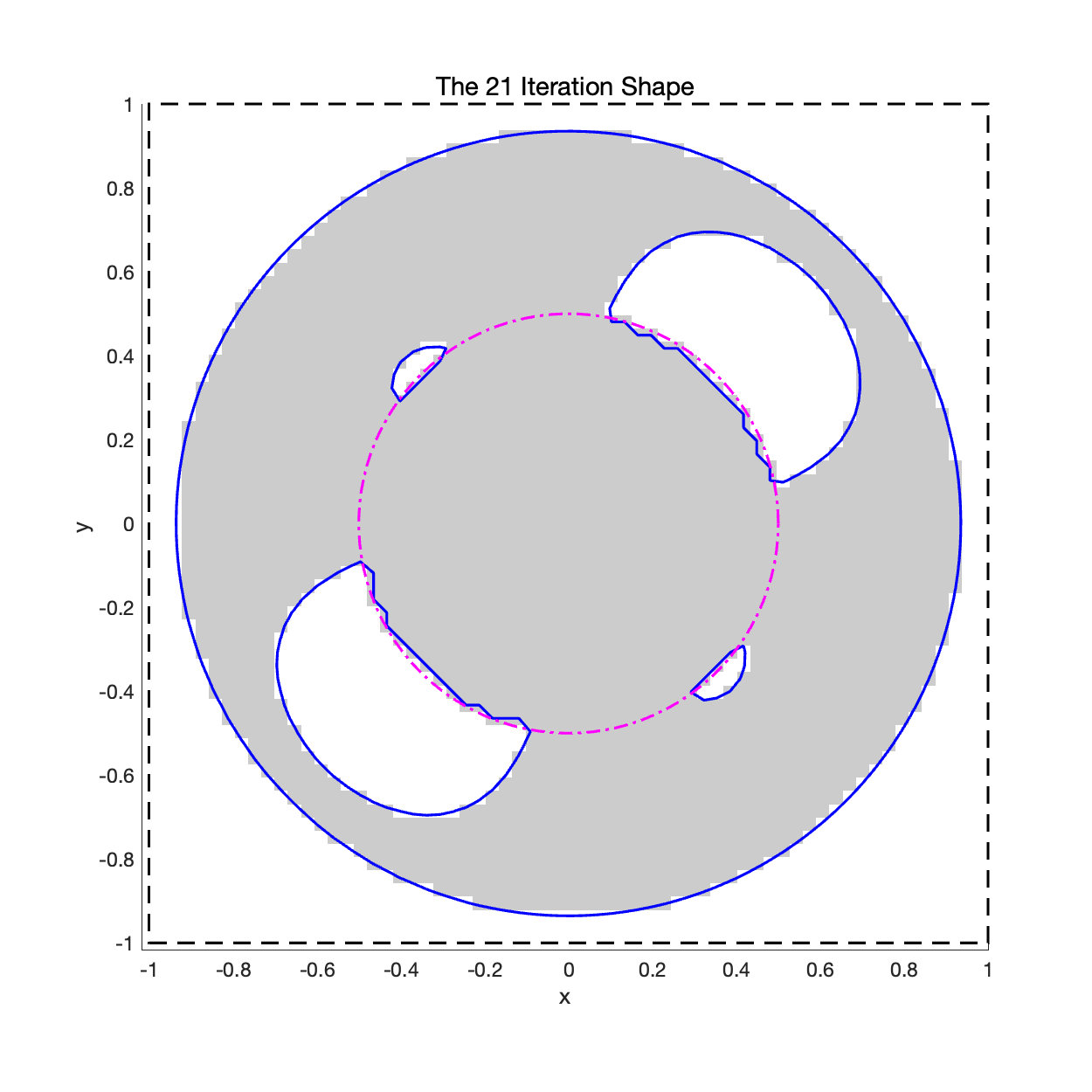}
\end{minipage}
\qquad
\begin{minipage}{0.325\linewidth}
\includegraphics[width=0.9\textwidth]{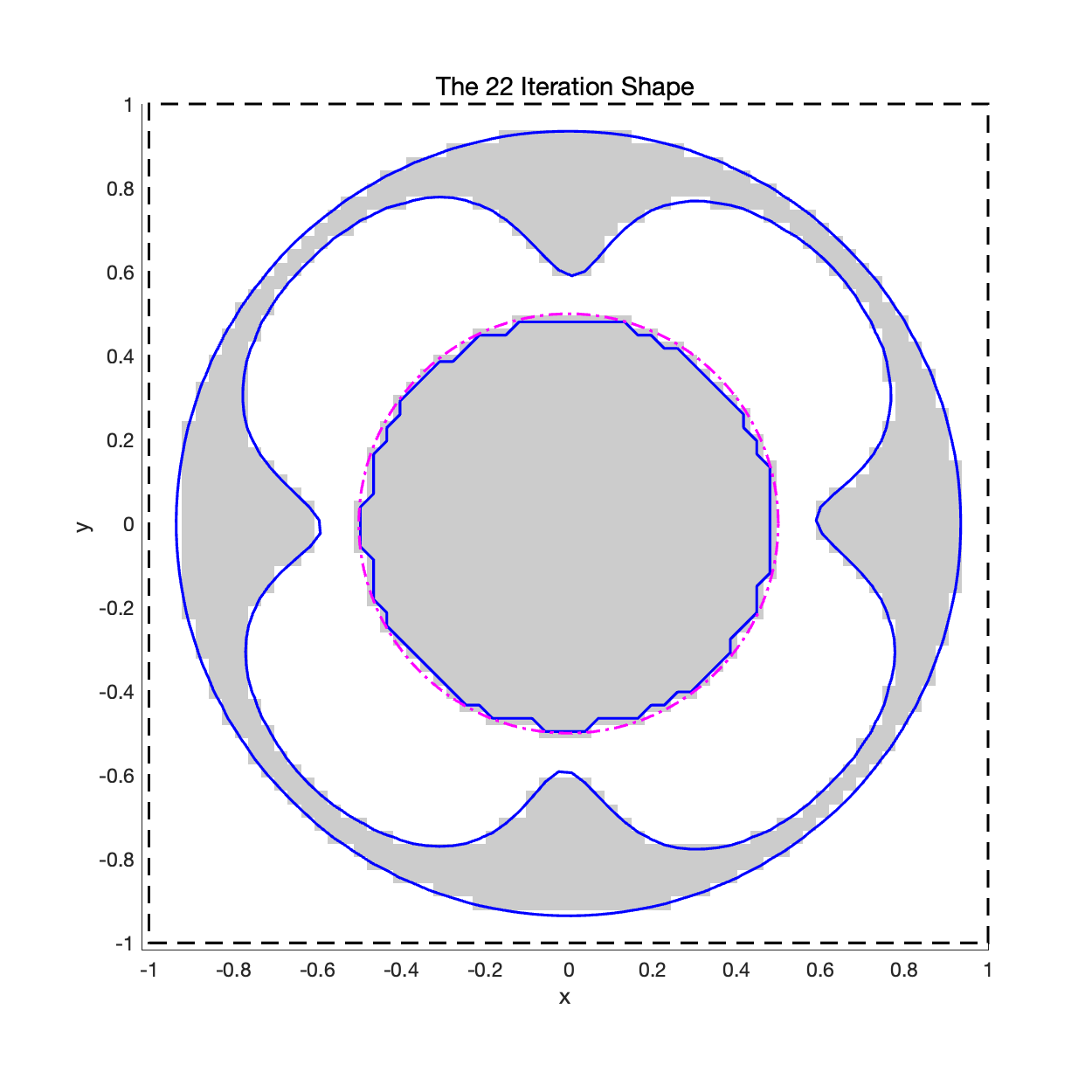}
\end{minipage}
\begin{minipage}{0.325\linewidth}
\includegraphics[width=0.9\textwidth]{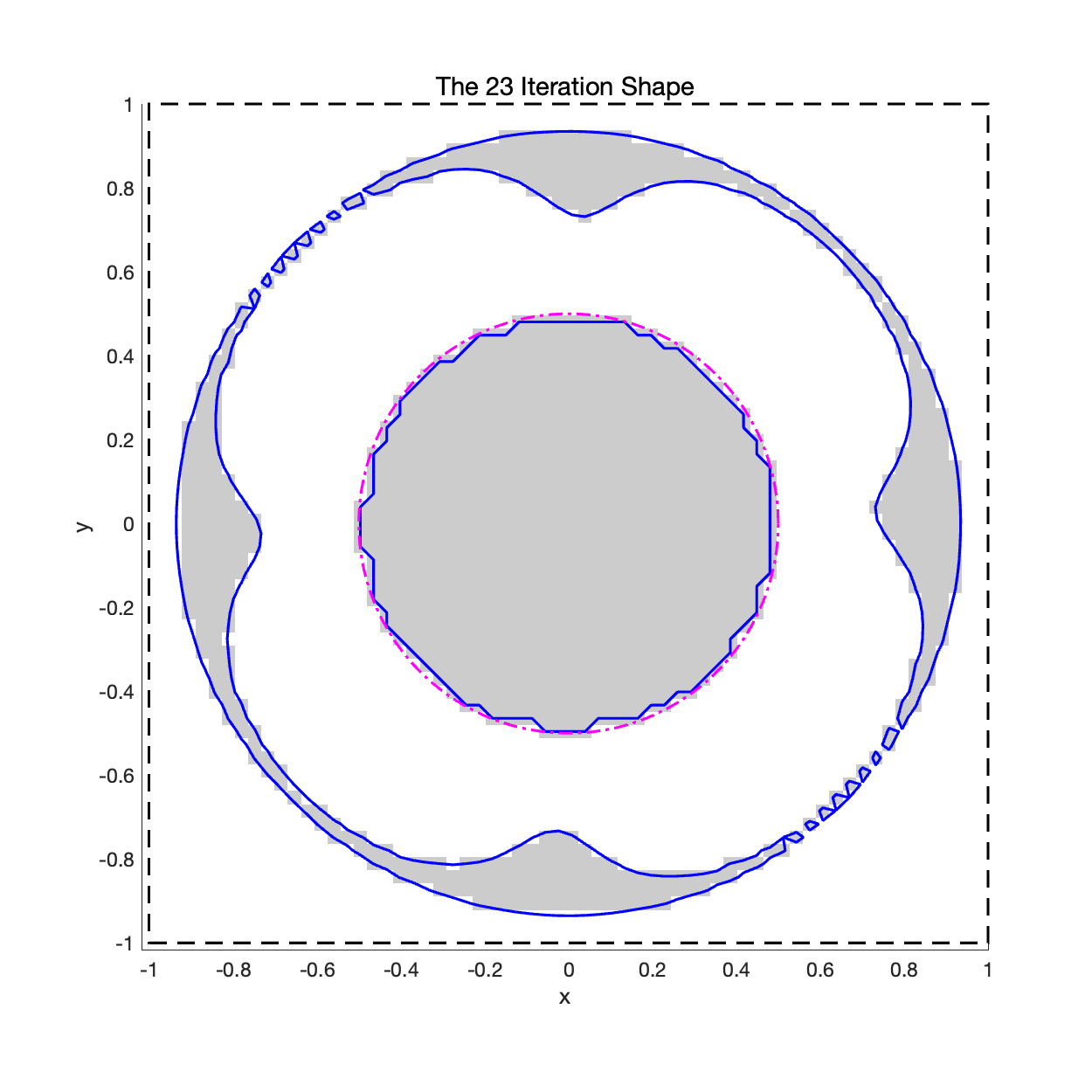}
\end{minipage}
\begin{minipage}{0.325\linewidth}
\includegraphics[width=0.9\textwidth]{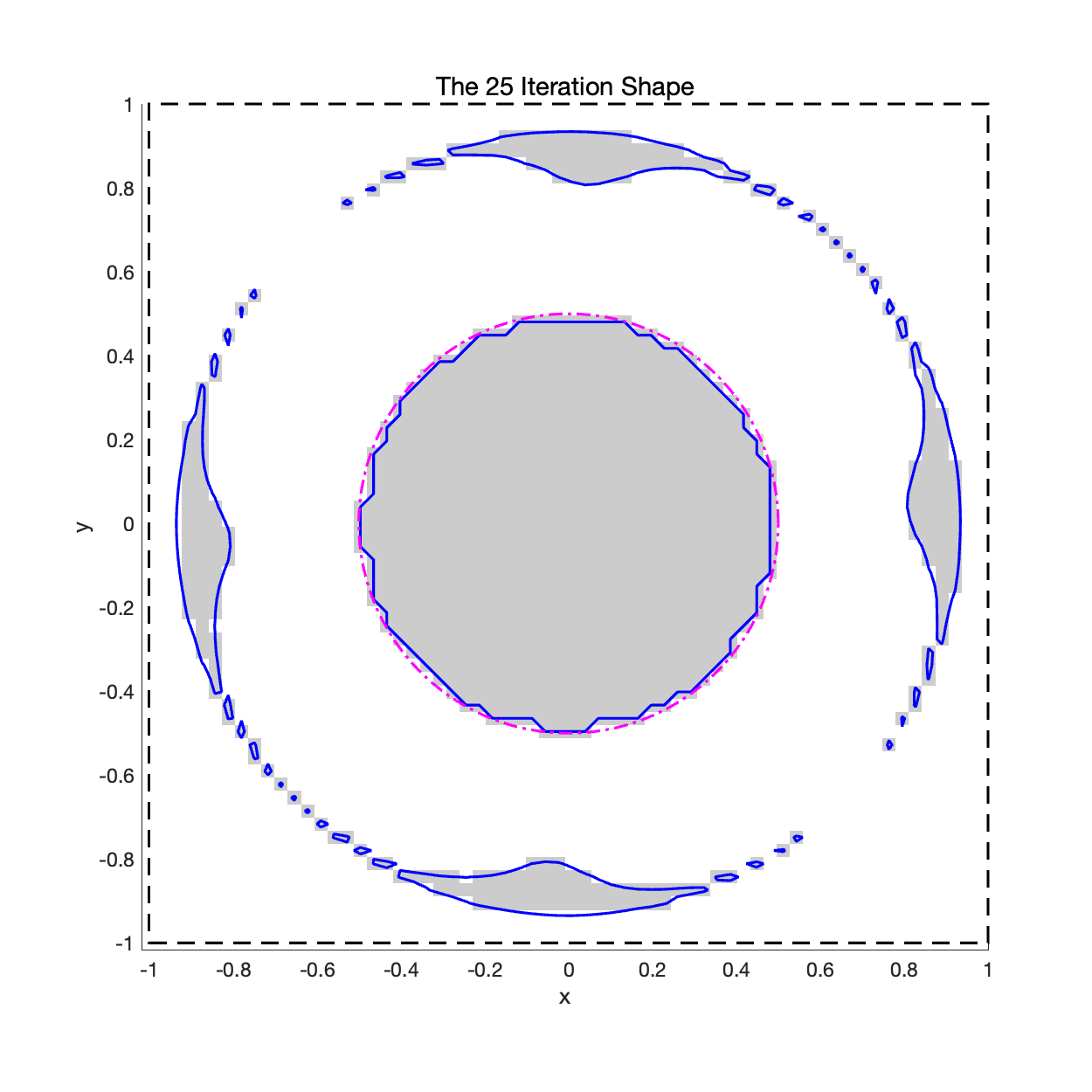}
\end{minipage}
\caption{The iterative shapes with boundary of steps 7, 8, 21, 22, 23 and 25.}
\label{fig:mediumshape3}
\end{center}
\end{figure}

\begin{figure}[H]
 \begin{center}
 \includegraphics[width=0.9\textwidth]{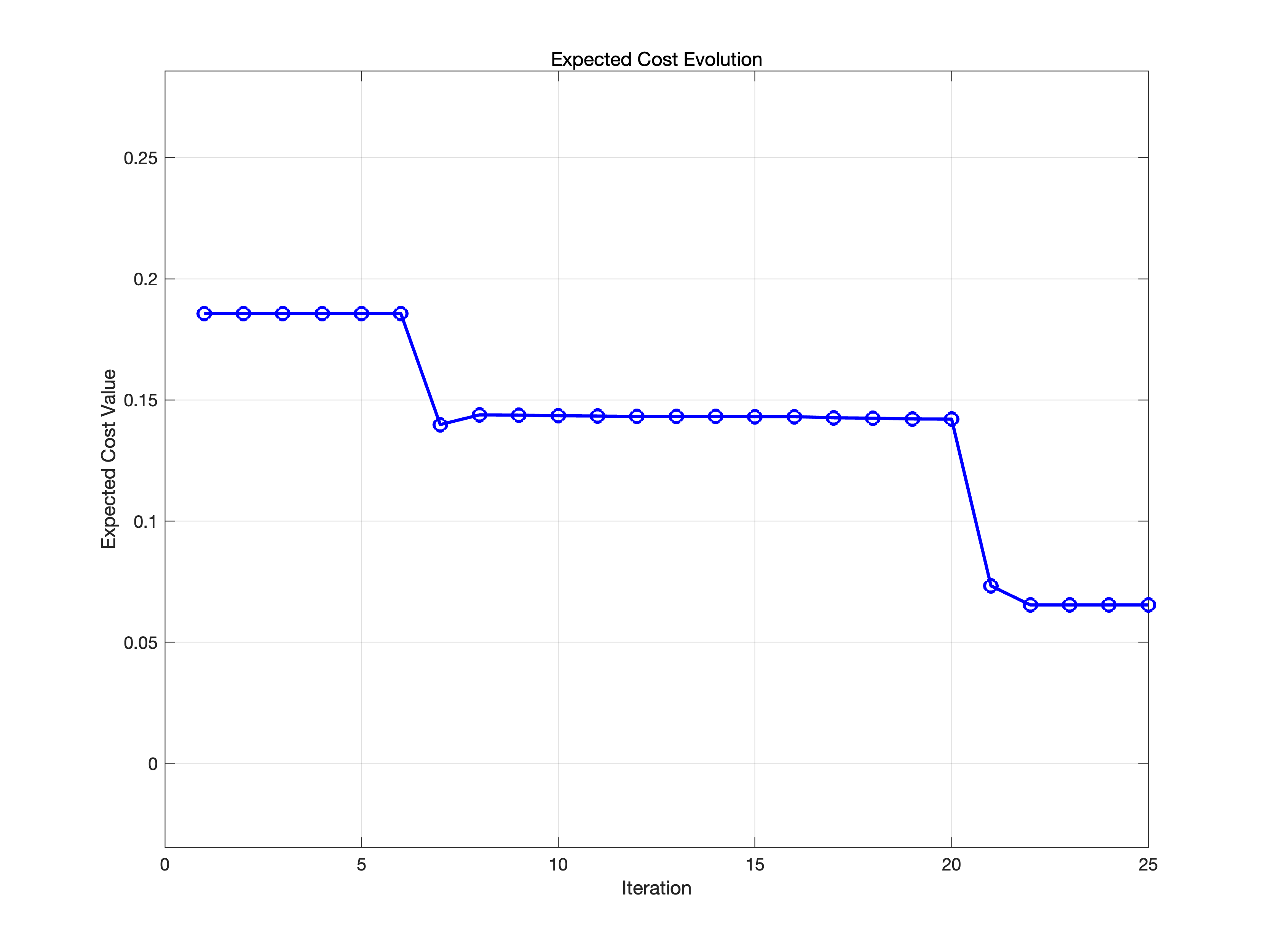}
 \caption{The variations of expected cost cure with iterations of Example \ref{ex2}.}
 \label{fig:ex3expectedcost}
 \end{center}
 \end{figure}

From the first three examples, it can be seen that regardless of whether region $O$ is a square region or a circular region, and regardless of whether the right-hand term is a constant function or not, the final shape  basically satisfies the theory that $O$ is the optimal shape, with only a small noise disturbance at the boundary. It can be seen from the change process that the shape change in the later stage of the iteration is relatively small.

\begin{example}\label{ex4}
For this example, we use a different objective functional, we restrict the integral domain is $K$, that is 
\begin{equation}
j(x,\omega,u(x,\omega))=\frac{1}{2}\mathbb{E}\left[\int_{K}(u-u_{d})^2dx\right]
\end{equation}
with  $u_{d}=-(x_{1}-\frac{1}{4})^2-(x_{2}-\frac{1}{4})^2+\frac{1}{25}$, or, equivalently, the objective functional is 
\begin{equation*} j(x,\omega,u(x,\omega))=\frac{1}{2}\mathbb{E}\left[\int_{D}H_{\epsilon}(g)\left(u-u_{d}\right)^2dx\right].
\end{equation*} 
The initial shape function is 
\begin{equation*} g^{0}(x)=\min\limits\left\{1-(x_{1}^2+x_{2}^{2}),x_{1}^{2}+x_{2}^{2}-\frac{1}{64},(x_{1}-\frac{1}{2})^{2}+x_{2}^{2}-\frac{1}{16}\right\},
\end{equation*} 
the force function $f=2$ and  the  restriction $g\geq 0$ in $O$ is not imposed. In the algorithm, we use the reduced descent direction  $d=-\mathbb{E}\left[\frac{1}{2}\left(u_{\epsilon}-u_{d}\right)^2+\frac{1}{\epsilon}u_{\epsilon}z\right]$.
\end{example}
We first show the initial shape in Figure~\ref{fig:initialshape}.
\begin{figure}[ht]
\begin{center}
\includegraphics[width=0.9\textwidth]{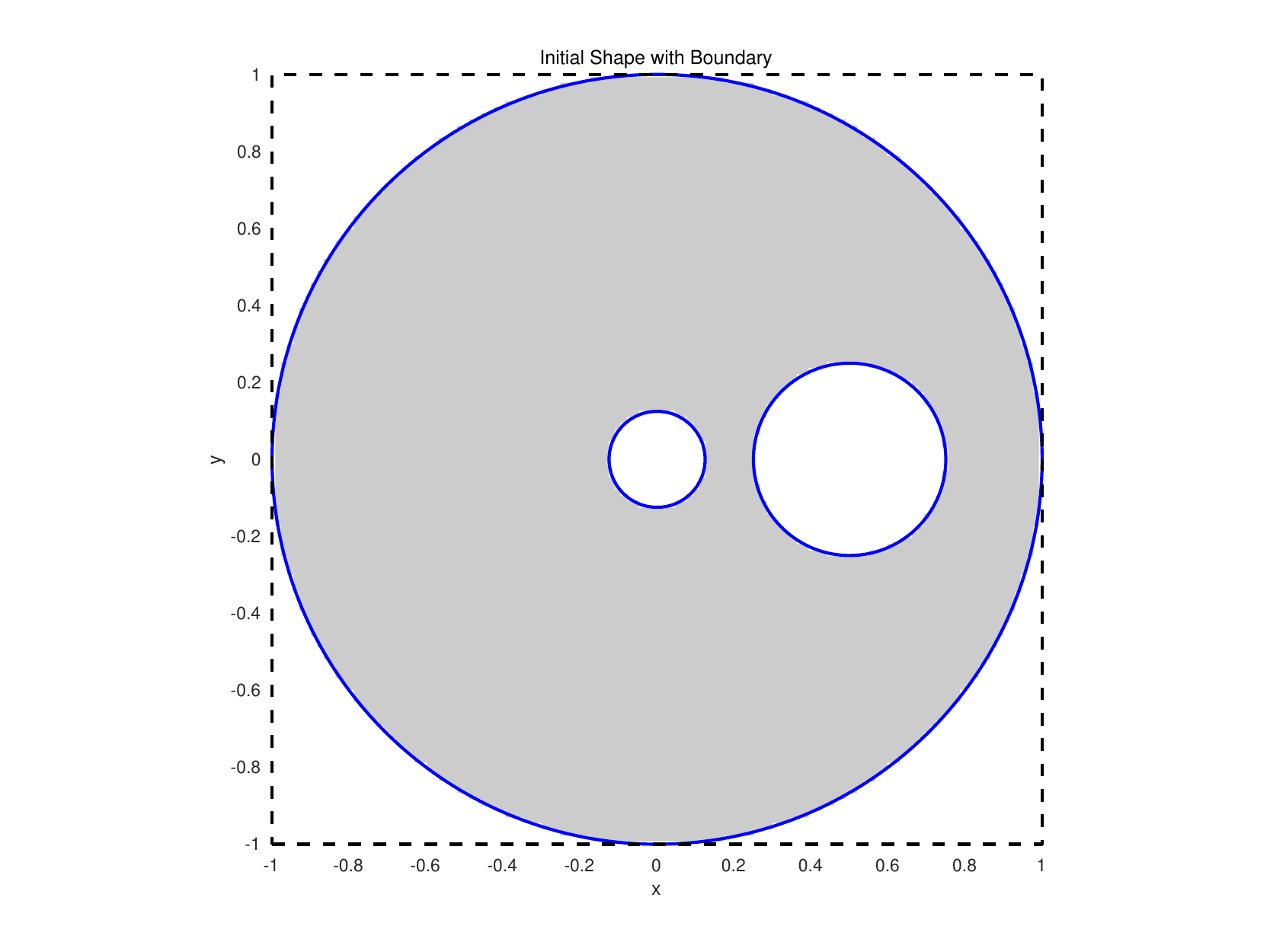}
\caption{The initial shape with boundary of Example~\ref{ex4}.}
\label{fig:initialshape}
\end{center}
\end{figure}

Then geometries of the intermediate step iteration are illustrated in Figure~\ref{fig:mediumshape}, where the  gray area stands for the region $g\geq 0$,  blue line denotes the boundary $g=0$ and the cavity is the region $g\leq 0$. 
 Figure~\ref{fig:expectedcost} show  the evolution process of expected costs with iterations. When the algorithm stops,  the final  absolute change of cost  value $|j(\bf{g}^{k+1})-j(\bf{g}^k)|$ is $9.1697e-13$, and the expected cost  value is $3.5e-5$. Actually, from the objective form, we can get the conclusion that the optimal shape is the empty set. From the perspective of the evolution process of shape, they are evolving towards an empty set.
\begin{figure}[H]
\begin{center}
\begin{minipage}{0.325\linewidth}
\includegraphics[width=0.9\textwidth]{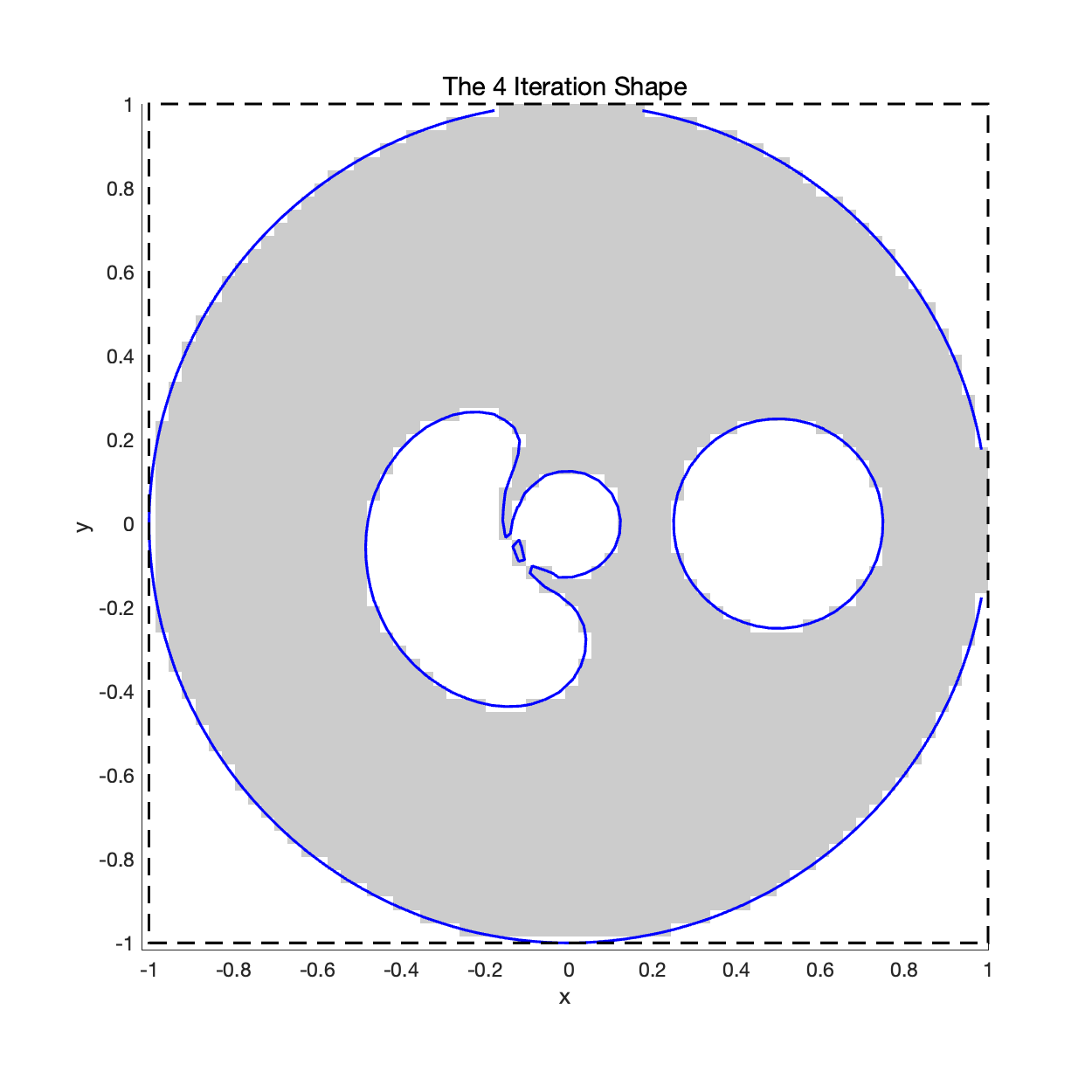}
\end{minipage}
\begin{minipage}{0.325\linewidth}
\includegraphics[width=0.9\textwidth]{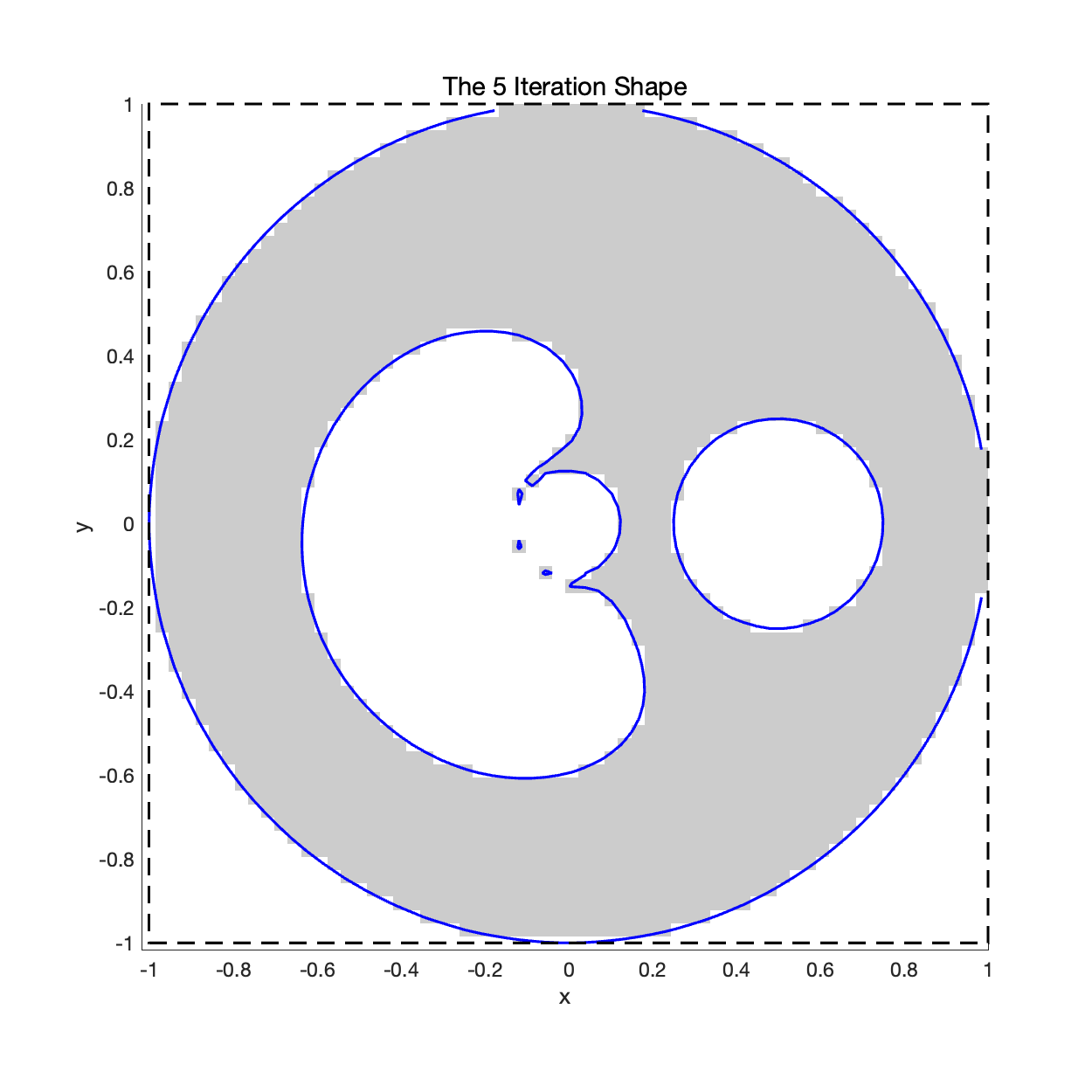}
\end{minipage}
\begin{minipage}{0.325\linewidth}
\includegraphics[width=0.9\textwidth]{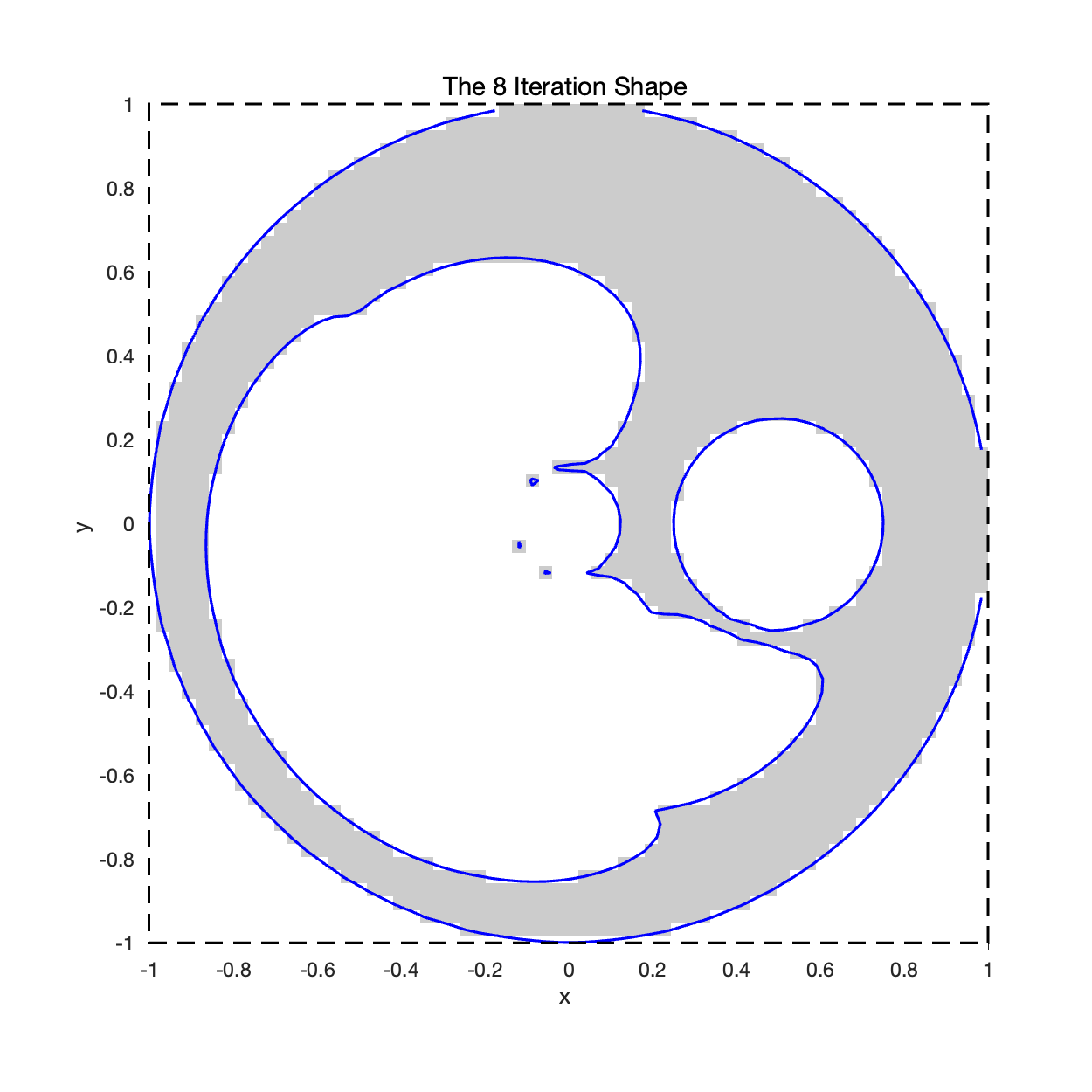}
\end{minipage}
\qquad
\begin{minipage}{0.325\linewidth}
\includegraphics[width=0.9\textwidth]{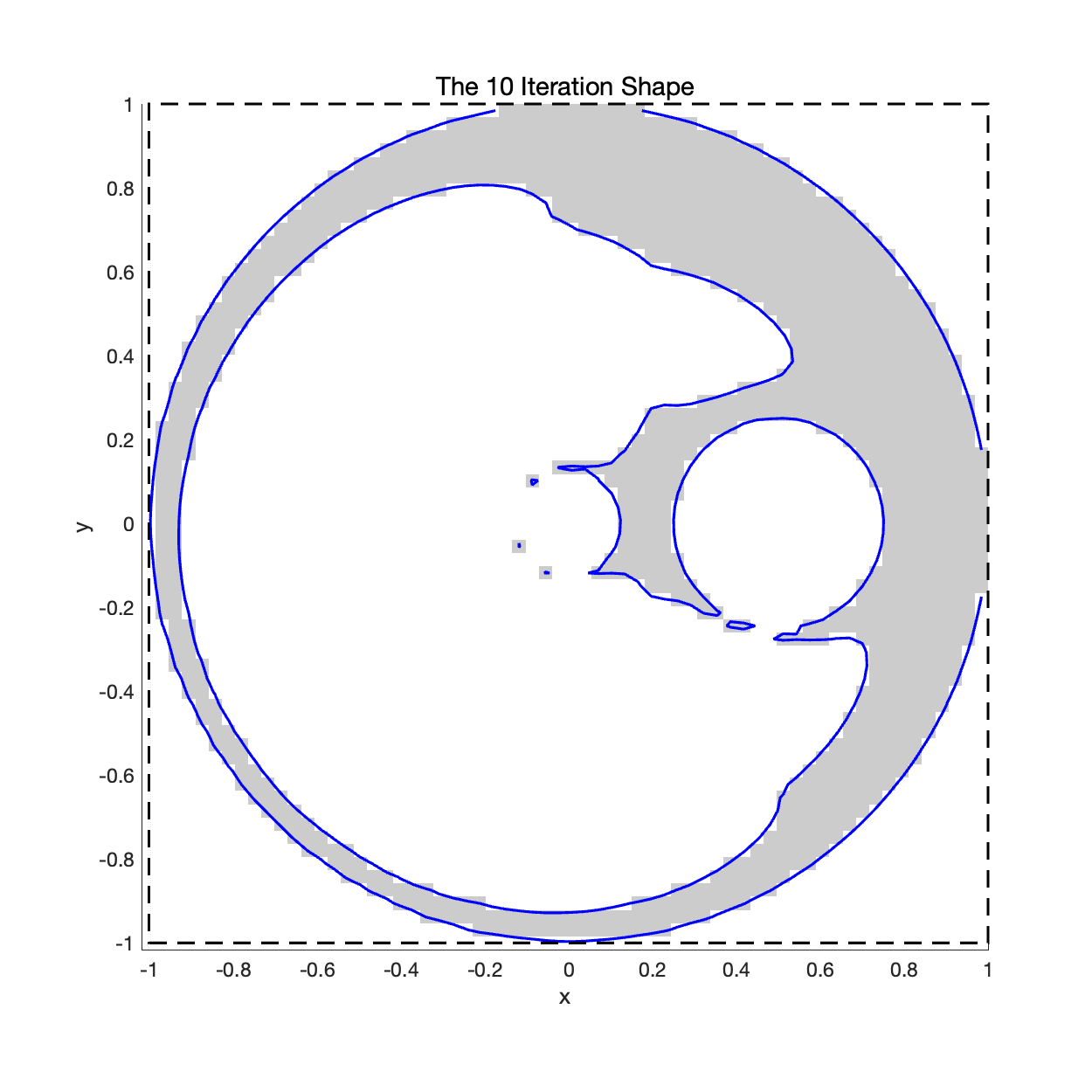}
\end{minipage}
\begin{minipage}{0.325\linewidth}
\includegraphics[width=0.9\textwidth]{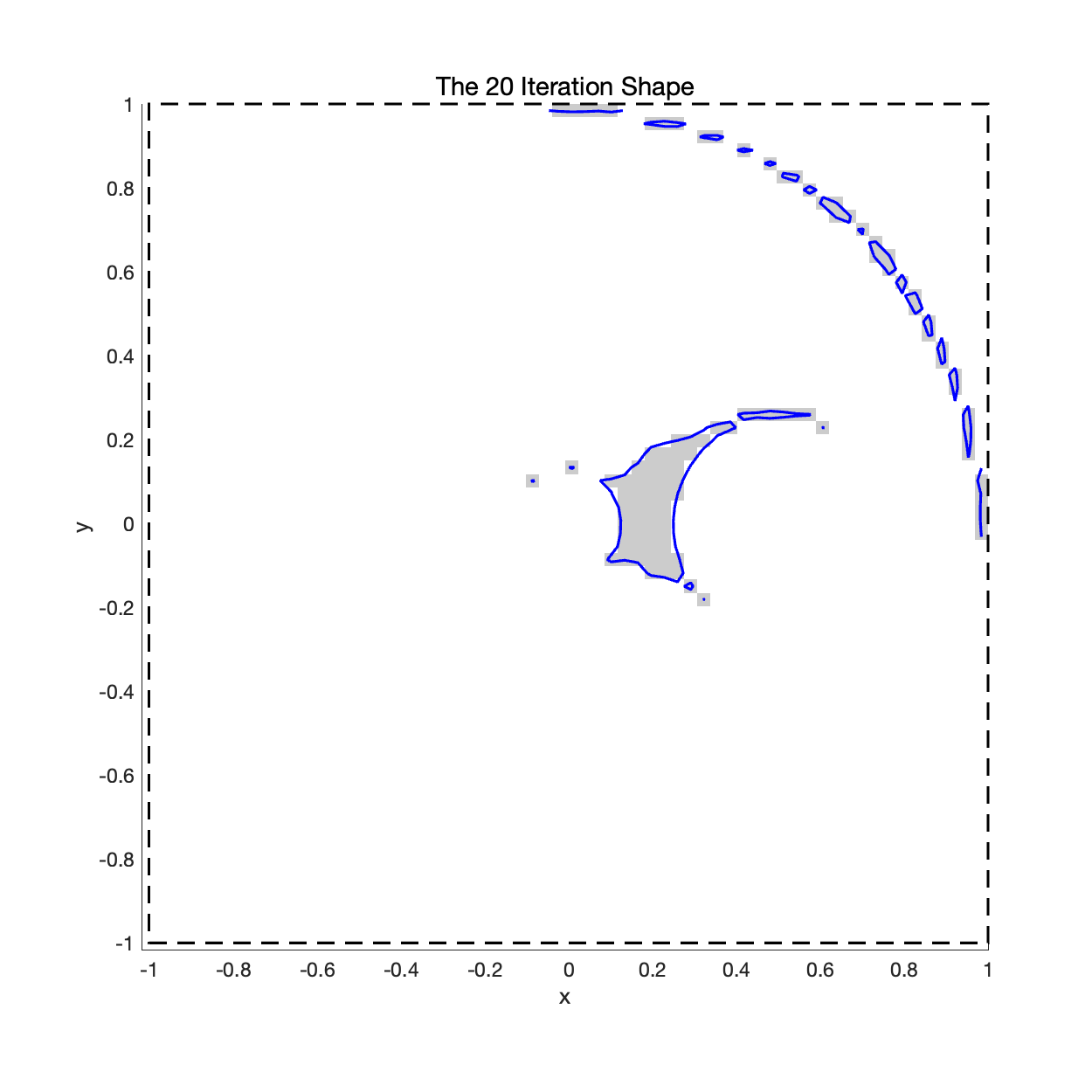}
\end{minipage}
\begin{minipage}{0.325\linewidth}
\includegraphics[width=0.9\textwidth]{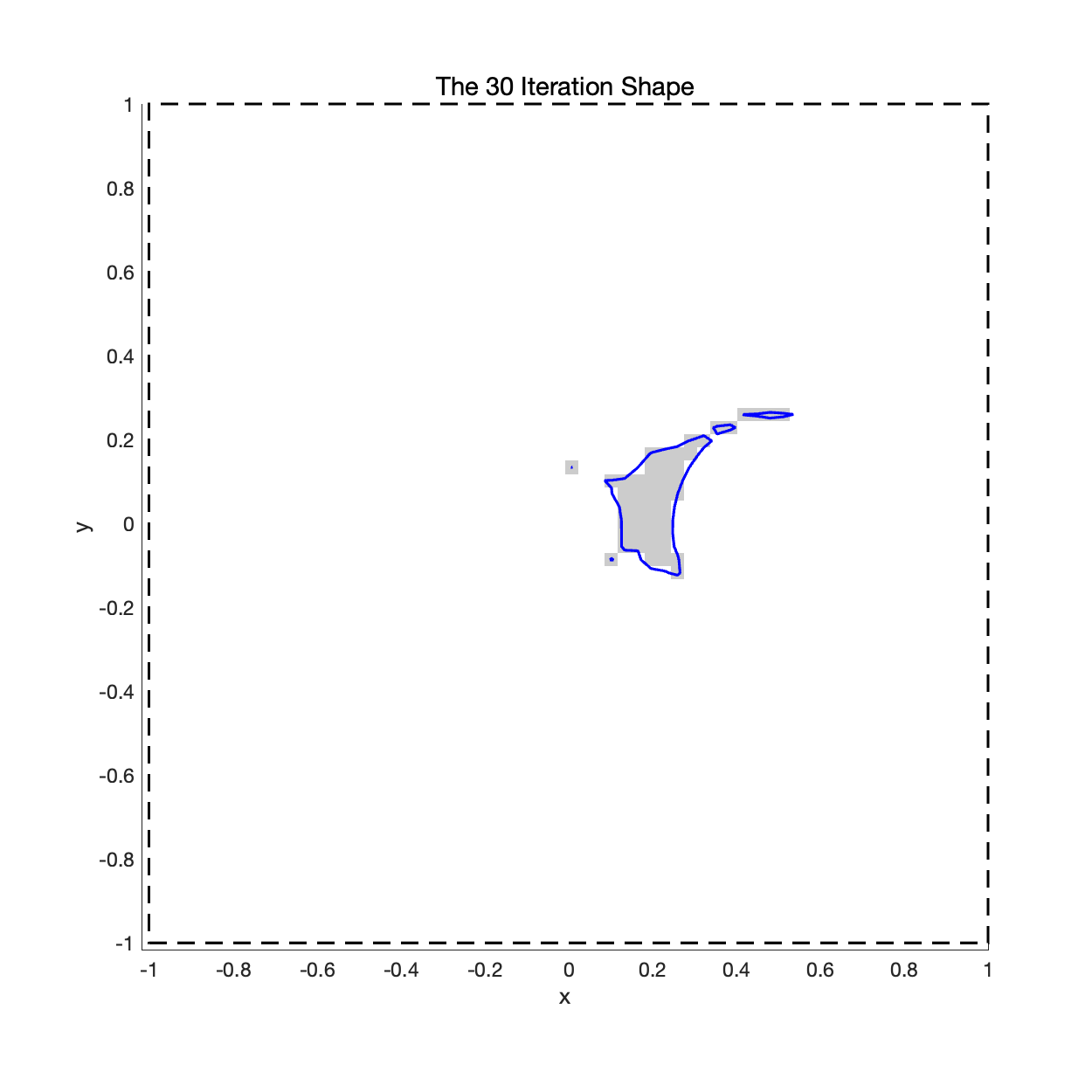}
\end{minipage}
\caption{The iterative shapes with boundary of steps 4, 5, 8, 10, 20 and 30.}
\label{fig:mediumshape}
\end{center}
\end{figure}
\begin{figure}[H]
\begin{center}
\includegraphics[width=0.9\textwidth]{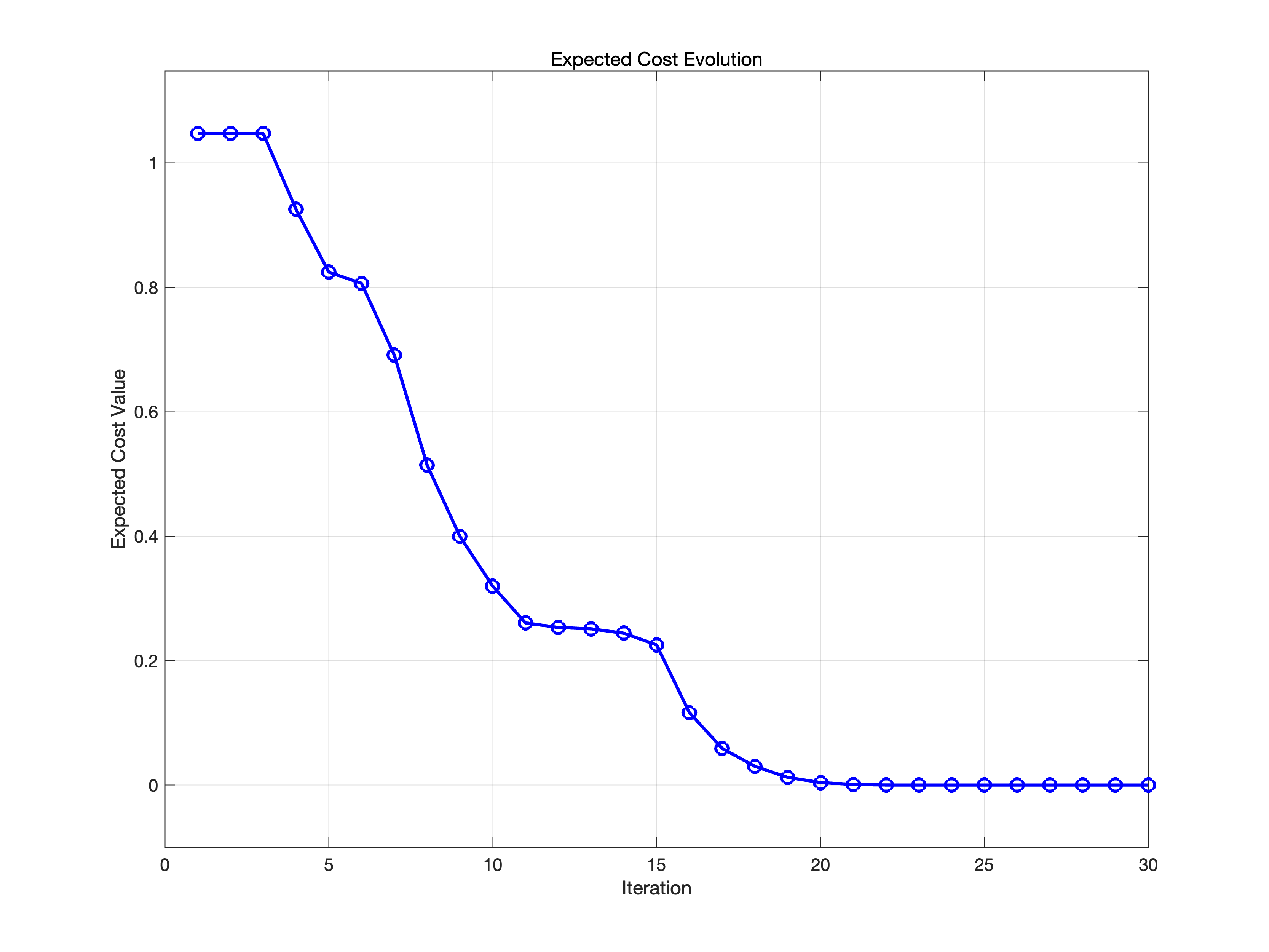}
\caption{The variations of expected cost cure with iterations of Example \ref{ex4}.}
\label{fig:expectedcost}
\end{center}
\end{figure}

\section{Conclusion}\label{sec5}
This paper presents an efficient discretize-then-optimize framework for elliptic PDE-constrained stochastic shape optimization with random diffusion coefficients. Element-wise stochasticity creates high-frequency noise requiring finite element method coupled with Monte Carlo method, where Monte Carlo method is applied to parallel compute the samples information.  While demonstrated on diffusion-based shape optimization, our approach can extend to any PDE-constrained optimization with uncertain parameters. 

\section*{Acknowledgements}
The author would like to thank Professor Tiba Dan for helpful discussions and supports to the completion of this paper. The work of X. Pang was supported by Science Foundation of Hebei Normal University  No. L2022B30 and China Scholarship Council No. 202308130193. The whole work was also funded by the Natural Science Foundation of Hebei Province No. A2023205045.

\end{document}